\documentclass[11pt]{article}
\usepackage{amsmath}
\usepackage{amssymb}
\usepackage{amsthm}
\usepackage{enumerate}
\usepackage[colorlinks]{hyperref}
\usepackage{enumitem}
\usepackage[margin=1in]{geometry}
\usepackage{stmaryrd}
\usepackage{mathtools}
\usepackage{tikz-cd}

\usepackage{tikz}
\usetikzlibrary{calc,angles,positioning,intersections,quotes,decorations.markings,shapes.geometric}
\usepackage{tkz-euclide}
\tikzset{>=latex}

\usepackage[style=numeric-comp,giveninits=true,url=false,maxbibnames=5,backend=bibtex]{biblatex}
\bibliography{references.bib}

\theoremstyle{plain}
\newtheorem{theorem}{Theorem}[section]
\newtheorem{lemma}[theorem]{Lemma}
\newtheorem{proposition}[theorem]{Proposition}
\newtheorem{corollary}[theorem]{Corollary}

\theoremstyle{definition}
\newtheorem{definition}[theorem]{Definition}

\newtheorem{remark}[theorem]{Remark}
\newtheorem{notation}[theorem]{Notation}

\DeclareMathOperator{\dv}{div}

\DeclareMathOperator{\Tr}{Tr}
\DeclareMathOperator{\Hess}{Hess}
\DeclareMathOperator{\Ric}{Ric}
\DeclareMathOperator{\diam}{diam}
\DeclareMathOperator{\df}{def}
\newcommand\region{U}
\newcommand\ds{ds}
\newcommand\distcurv{(\kappa\omega)_{\rm dist}}

\DeclareMathOperator\End{End}
\newcommand\A{\mathcal A}

\newcommand\dd[2][]{\frac{d{#1}}{d{#2}}}

\newcommand{\textoverline}[1]{$\overline{\mbox{#1}}$}

\colorlet{dgreen}{black}
\colorlet{blue}{black}

\begin{document}

\title{Finite element approximation of the Levi-Civita connection and its curvature in two dimensions}

\author{Yakov Berchenko-Kogan\thanks{\noindent Department of Mathematics, Pennsylvania State University, \href{yashabk@psu.edu}{yashabk@psu.edu}} \and Evan S. Gawlik\thanks{\noindent Department of Mathematics, University of Hawai`i at M\textoverline{a}noa, \href{egawlik@hawaii.edu}{egawlik@hawaii.edu}}}

\date{}

\maketitle

\begin{abstract}
We construct finite element approximations of the Levi-Civita connection and its curvature on triangulations of oriented two-dimensional manifolds.  Our construction relies on the Regge finite elements, which are piecewise polynomial symmetric $(0,2)$-tensor fields possessing single-valued tangential-tangential components along element interfaces.  When used to discretize the Riemannian metric tensor, these piecewise polynomial tensor fields do not possess enough regularity to define connections and curvature in the classical sense, but we show how to make sense of these quantities in a distributional sense.  We then show that these distributional quantities converge in certain dual Sobolev norms to their smooth counterparts under refinement of the triangulation.  We also discuss projections of the distributional curvature and distributional connection onto piecewise polynomial finite element spaces.  We show that the relevant projection operators commute with certain linearized differential operators, yielding a commutative diagram of differential complexes. 
\end{abstract}

\section{Introduction} \label{sec:intro}

The finite element method is used ubiquitously to approximate solutions to partial differential equations in Euclidean space, but it sees relatively limited use in Riemannian geometry.  The goal of this paper is to lay down foundations for computing two quantities of interest in Riemannian geometry---the Levi-Civita connection and its curvature---with finite elements.  We focus on the two-dimensional setting.  

Our construction relies on the Regge finite elements, which are a recently developed family of finite elements for discretizing symmetric $(0,2)$-tensor fields on simplicial triangulations~\cite{li2018regge,christiansen2011linearization,regge1961general}.  When used to discretize the Riemannian metric tensor, these piecewise polynomial tensor fields do not possess enough regularity to define connections and curvature in the classical sense.  We show in this paper how to make sense of these quantities in a distributional sense.  Importantly, these distributional quantities converge to their smooth counterparts under refinement of the triangulation in a sense that we make precise in Section~\ref{sec:conv}.  The rates of convergence depend on the polynomial degree of the approximate metric tensor, with higher polynomial degrees yielding higher rates of convergence.

To be more concrete, let us briefly describe the Regge finite elements~\cite{li2018regge,christiansen2011linearization,regge1961general}.  Given a triangulation $\mathcal{S}$ of an oriented manifold of dimension $d$, the lengths of all of the edges in $\mathcal{S}$ determine a piecewise constant Riemannian metric $g$ on $\mathcal{S}$.  This metric automatically possesses the following continuity property: $g$ has single-valued tangential-tangential components on every $(d-1)$-dimensional simplex in $\mathcal{S}$.  The metric $g$ is an example of a tensor field belonging to the lowest-order Regge finite element space.  More generally, for an integer $r \ge 0$, the Regge finite element space of order $r$ consists of symmetric $(0,2)$-tensor fields on $\mathcal{S}$ that are piecewise polynomial of degree at most $r$ and obey the same tangential-tangential continuity constraint as above.  Often the integer $r$ is unimportant, and we will simply be concerned with the space of piecewise smooth symmetric $(0,2)$-tensor fields with tangential-tangential continuity across $(d-1)$-dimensional faces.  We call elements of this space \emph{Regge metrics} if they are positive definite everywhere. 

Obviously, the scalar curvature of a Regge metric $g$ is not well-defined in the classical sense, unless attention is restricted to the interior of a $d$-simplex in $\mathcal{S}$.  However, there is a natural way to interpret the scalar curvature (more precisely, the scalar curvature times the volume form) of $g$ in a distributional sense when $g$ is piecewise constant.  One considers a linear combination of Dirac delta distributions supported on $(d-2)$-simplices $z$, each weighted by the \emph{angle defect} at $z$ times the volume of $z$.  The angle defect measures the failure of the dihedral angles incident at $z$ to sum to $2\pi$.  This definition of scalar curvature was posited by Regge in his discrete theory of relativity~\cite{regge1961general} and has since been given various justifications~\cite{cheeger1984curvature,christiansen2011linearization,christiansen2013exact}.

The first aim of this paper is to study a generalization of Regge's definition of scalar curvature to piecewise polynomial Regge metrics in dimension $d=2$.  It turns out that the appropriate generalization is a distribution with three contributions: the scalar curvature within each triangle, the jump in the geodesic curvature across each edge, and the angle defect at each vertex.  Such a definition has been mentioned in the discrete differential geometry {\color{blue}and geometric analysis} literature (see for instance~\cite[p. 6]{sullivan2008curvatures} {\color{blue}and~\cite[Corollary 3.1]{strichartz2020defining}}), but, to our knowledge, no efforts have been made to understand its convergence until now, {\color{blue}and it does not appear to have been mentioned in~\cite{li2018regge}.}

The second aim of this paper is to give meaning to the Levi-Civita connection associated with a Regge metric $g$.  We again restrict our attention to dimension $d=2$, where it is possible to encode the Levi-Civita connection locally with a scalar-valued one-form.  We construct such a one-form {\color{blue}using certain rotation angles associated with parallel transport across edges in the triangulation}.  Our construction leads to a distributional one-form $\Gamma_{\rm dist}$ whose distributional exterior \emph{co}derivative is equal to the distributional curvature of $g$ discussed above.  As such, this one-form can be regarded as a distributional version of the Hodge star of the corresponding connection one-form from the smooth setting.  When $g$ is piecewise constant, $\Gamma_{\rm dist}$ is a distribution supported on edges.  This aligns with a common viewpoint in discrete exterior calculus~\cite{hirani2003discrete,desbrun2005discrete}, where discrete connections on two-dimensional triangulations are often regarded as discrete dual one-forms~\cite{leok2005discrete,crane2010trivial}; the discrete Hodge star of such a discrete dual one-form is naturally associated with (primal) edges of the triangulation, just as our distributional connection is.  In fact, relative to an appropriate orthonormal frame, the distributional connection one-form that we construct encodes a parallel transport operator with the following properties. Along any curve $\mathcal{C}$ that lies entirely in the interior of a triangle, the parallel transport operator along $\mathcal{C}$ coincides with the smooth Levi-Civita parallel transport operator.  If $\mathcal{C}$ crosses an edge, then the tangential and normal components of any vector are preserved during parallel transport across the edge.  This is a widely used notion of parallel transport on triangulations~\cite{leok2005discrete,crane2010trivial,li2018regge,christiansen2011linearization}.

It turns out that a great deal of information about a Regge metric's distributional curvature and distributional connection can be gleaned from studying their evolution under deformations of the metric.  For one thing, doing so allows us to show that the distributional curvature operator described above is (infinitesimally) consistent; its linearization around a given Regge metric $g$ is precisely the linearized curvature operator, interpreted in a distributional sense.  Christiansen~\cite[Proposition 2]{christiansen2011linearization} showed this in the special case where $g$ is Euclidean, and the second author showed this when $g$ is a piecewise constant Regge metric~\cite[Lemma 3.3]{gawlik2019high}.  The present paper shows this for arbitrary Regge metrics $g$.  The calculation is more involved than in the former papers, since, among other things, one must linearize the jumps in the geodesic curvature across edges of the triangulation and make use of several non-Euclidean integration by parts identities. 

The result of this calculation becomes especially illuminating when expressed in finite element parlance.  As we show in Theorem~\ref{thm:main}, the formula for the linearization of the distributional curvature of a Regge metric $g$ is expressible in terms of a bilinear form that appears in the Hellan-Herrmann-Johnson (HHJ) finite element method~\cite{babuska1980analysis,arnold1985mixed,brezzi1977mixed,braess2018two,braess2019equilibration,arnold2020hellan,pechstein2017tdnns,chen2018multigrid}.  It is not the familiar Euclidean version of this bilinear form, but rather one obtained by replacing the Euclidean metric by $g$.  

This link with the HHJ method plays a central role in our analysis.  It reveals that the second author's prior work~\cite{gawlik2019high} on curvature approximation---where an integral of the HHJ bilinear form was used to \emph{define} the curvature of piecewise polynomial Regge metrics---is directly applicable to our setting, because the approximate curvature defined there is (somewhat fortuitously) equivalent to the one studied here.  This allows us to leverage the analysis in~\cite{gawlik2019high} to deduce the convergence of the distributional curvature under refinement.  In this analysis, the evolution of geometric quantities under metric deformations plays a key role.  Roughly speaking, to bound the error in the curvature approximation, one studies the evolution of the error along a one-parameter family of Regge metrics emanating from the Euclidean metric, where the error is zero.

Although we have chosen to focus on defining distributional connections and distributional curvature in this paper, it is worth noting that both the distributional curvature and the distributional connection can, if desired, be projected onto finite element spaces.   Doing so produces piecewise polynomial quantities that are computable using standard finite element assembly routines.  As we show in Section~\ref{sec:poly}, appropriate finite element spaces to use for the curvature and connection (when the Regge metric $g$ is piecewise polynomial of degree at most $r$) are the spaces $\mathcal{P}_{r+1}\Lambda^0$ and $\mathcal{P}_{r+1}^-\Lambda^1$ from finite element exterior calculus~\cite{arnold2006finite,arnold2010finite}.  These choices, which correspond to continuous Lagrange finite elements and two-dimensional N\'{e}d\'{e}lec finite elements of the first kind, respectively, are guided by commutative diagrams of differential complexes; see Section~\ref{sec:poly}.

Note that elsewhere in the literature, one can find notions of discrete connections and curvature on simplicial triangulations that differ from ours in important ways.
For example,~\cite{berwick2021discrete} and~\cite{christiansen2012simplicial} associate a parallel transport map to each (primal) edge in the triangulation, which is interpreted as a map between vector spaces anchored at vertices.  As such, it leads to a notion of curvature that is associated with triangles rather than $(d-2)$-simplices.  Another approach~\cite{christiansen2019finite} associates a parallel transport map to every pair of simplices for which one member of the pair is a codimension-1 subsimplex of the other, leading to a notion of curvature that is associated with elements of the cubical refinement of the triangulation.  Similarly,~\cite{liu2016discrete} associates (in dimension $d=2$) a parallel transport map to every pair of incident simplices of arbitrary dimension, leading to a notion of curvature associated with triangles.  In contrast, the viewpoint we adopt in this paper is more closely aligned with the discrete exterior calculus viewpoint in~\cite{leok2005discrete,crane2010trivial} and with the classical viewpoint that curvature is concentrated on $(d-2)$-simplices in the piecewise flat setting.

This paper is organized as follows.  We start by deriving formulas for the evolution of various geometric quantities under deformations of the metric in Section~\ref{sec:evolution}.  There, the focus is on smooth Riemannian metrics.  We then turn our attention toward Regge metrics in Section~\ref{sec:distcurv} and define the distributional curvature of a Regge metric.  We use the results of Section~\ref{sec:evolution} to compute the linearization of the distributional curvature in Section~\ref{sec:linearization}.  The formula for the linearization, together with the calculations from Section~\ref{sec:evolution}, {\color{blue}play a role in Section~\ref{sec:distconn}, where we introduce and study the properties of the distributional Levi-Civita connection.}  We study the convergence of the distributional curvature and distributional connection to their smooth counterparts under refinement in Section~\ref{sec:conv}.  We discuss projections of the distributional curvature and distributional connection onto piecewise polynomial finite element spaces in Section~\ref{sec:poly}.  We show there that the relevant projection operators commute with certain linearized differential operators, yielding a commutative diagram of differential complexes.

\section{Evolution of geometric quantities} \label{sec:evolution}

\paragraph{Notation.}

Let $M$ be a smooth oriented manifold of dimension $d$.  We use $TM$ and $T^*M$ to denote the tangent and cotangent bundles of $M$, respectively.  The set of differential $k$-forms on $M$ is denoted $\Lambda^k(T^*M)$, and the endomorphism bundle of $TM$ is denoted $\End(TM)$.

Let $g$ be a smooth Riemannian metric on $M$.
We use $\omega$ to denote the volume form on $M$ induced by $g$.  The Levi-Civita connection associated with $g$ is denoted $\nabla$.  Thus, if $\sigma$ is a $(p,q)$-tensor field, then the covariant derivative of $\sigma$ is the $(p,q+1)$-tensor field $\nabla \sigma$, and the covariant derivative of $\sigma$ along a vector field $X$ is the $(p,q)$-tensor field $\nabla_X \sigma$.  We use $\Tr \sigma$ to denote the contraction of $\sigma$ along the first two indices, using $g$ to raise or lower indices as needed.  We write $\dv \sigma = \Tr \nabla\sigma$ and $\Delta \sigma = \dv \nabla \sigma$.  The Riemannian Hessian of a scalar field $f$ is denoted $\Hess f = \nabla \nabla f$.

The pointwise inner product of two $(p,q)$-tensor fields $\sigma$ and $\rho$ with respect to $g$ is denoted $\langle \sigma, \rho \rangle_g$.  Their $L^2$-inner product over $M$ is $\langle \sigma, \rho \rangle_{g,M} = \int_M \langle \sigma, \rho \rangle_g \, \omega$.  Sometimes we suppress the subscript $g$ when the metric is clear from the context.

The Lie derivative of a $(p,q)$-tensor field $\sigma$ along a vector field $X$ is denoted $\mathcal{L}_X \sigma$.  If $X$ and $Y$ are two vector fields, then we denote their Lie bracket by $[X,Y] = \mathcal{L}_X Y$.  We also use $[u,v]$ to denote the commutator $uv-vu$ of two endorphisms $u$ and $v$, which we interpret pointwise if $u$ and $v$ vary spatially.

If $\alpha$ is a differential $k$-form, then its exterior derivative is the $(k+1)$-form $d\alpha$, its Hodge star is the $(d-k)$-form $\star \alpha$, its exterior coderivative is the $(k-1)$-form $d^* \alpha = (-1)^k \star^{-1} d \star \alpha$, and its contraction along a vector field $X$ is the $(k-1)$-form $i_X \alpha$.  The wedge product of two differential forms $\alpha$ and $\beta$ is denoted $\alpha \wedge \beta$.

In addition to using $i_X$ to denote the contraction along $X$, we use the letter $i$ for another purpose.  If $N$ is a submanifold of $M$, then $i_{M,N}$ denotes the inclusion $N \hookrightarrow M$, and $i_{M,N}^*$ denotes the pullback under this inclusion.

We use $\sharp$ and $\flat$ to denote the musical isomorphisms sending one-forms to vector fields and vice versa.  If $f$ is a scalar field, then its covariant derivative $\nabla f$ coincides with the one-form $df$, but we will frequently abuse notation by interpreting $\nabla f$ as either $df$ or $(df)^\sharp$ depending on the context.

We make occasional use of index notation to do calculations in coordinates.  We always follow the Einstein summation convention.  Thus, $\nabla_X f = X^i \nabla_i f = X_i \nabla^i f$, $\Delta f = \nabla^i \nabla_i f = \nabla_i \nabla^i f$, etc.

We use the letter $\delta$ to denote the Euclidean metric.

In our analysis, an important role will be played by the operator $S$ which sends a symmetric $(0,2)$-tensor field $\sigma$ to the symmetric $(0,2)$-tensor field
\[
S\sigma = \sigma - g \Tr\sigma.
\]

From this point forward, we restrict our attention to dimension $d=2$.

\paragraph{Outline.}

The goal of this section is to understand how various quantities associated with the metric $g$, like the curvature and Levi-Civita connection, evolve with time if $g$ is time-dependent.  Thus, we consider an evolving Riemannian metric $g(t)$ with time derivative 
\[
\sigma = \frac{\partial}{\partial t}g,
\] 
and we express all of our results in terms of $\sigma$.
We use dots to denote differentiation with respect to $t$.

We calculate the evolution of four quantities, working throughout in dimension $d=2$: the Gaussian curvature (Section~\ref{sec:kappavoldot}), the Levi-Civita connection (Section~\ref{sec:conndot}), the geodesic curvature of a curve in $M$ (Section~\ref{sec:klengthdot}), and the angle between two curves in $M$ (Section~\ref{sec:angledot}).  The results of these calculations will be used extensively when we study Regge metrics in Sections~\ref{sec:distcurv}-\ref{sec:poly}.

\subsection{Gaussian curvature evolution} \label{sec:kappavoldot}

We first study the evolution of the curvature two-form: the Gaussian curvature $\kappa$ (which is half the scalar curvature $R$) times the volume form $\omega$.  

\begin{proposition} \label{prop:kappavoldot}
If $g(t)$ is an evolving Riemannian metric on $M$ with time derivative $\sigma = \frac{\partial}{\partial t}g$, then the curvature 2-form $\kappa \, \omega$ satisfies
\begin{equation} \label{kappavoldot}
\frac{\partial}{\partial t} \left( \kappa \, \omega \right) = \frac{1}{2} (\dv \dv S \sigma) \, \omega,
\end{equation}
where $S\sigma = \sigma - g \Tr \sigma$.
\end{proposition}
\begin{proof}
We use the following well-known formula for the time derivative of $\kappa$~\cite[Lemma 2]{fischer1975deformations},~\cite[Equation 2.4]{chow2006hamilton}: 
\[
\dot{\kappa} = \frac{1}{2} \left( \dv \dv \sigma - \Delta \Tr \sigma - \langle \sigma, \Ric \rangle \right).
\]
Here, $\Ric$ denotes the Ricci tensor, which is simply $\kappa g$ in two dimensions.  Since $\Delta v = \dv \dv (gv)$ for any scalar field $v$, we can write
\[
\dot{\kappa} = \frac{1}{2} \left( \dv\dv(\sigma - g \Tr \sigma) - \kappa \Tr \sigma \right) = \frac{1}{2} \left( \dv \dv S \sigma - \kappa \Tr \sigma\right).
\]
On the other hand, we have~\cite[Equation 2.4]{chow2006hamilton}
\begin{equation} \label{voldot}
\dot{\omega} = \frac{1}{2} (\Tr \sigma) \, \omega,
\end{equation}
so
\[
\frac{\partial}{\partial t} \left( \kappa \, \omega \right) = \dot{\kappa} \, \omega + \frac{1}{2} \kappa (\Tr \sigma) \, \omega = \frac{1}{2} (\dv \dv S \sigma) \, \omega.
\]
\end{proof}

\subsection{Evolution of the Levi-Civita connection} \label{sec:conndot}

We now turn our attention to the evolution of the Levi-Civita connection $\nabla$.  
For this task, it will be convenient to focus on a region $\region \subseteq M$ on which the tangent bundle $T\region$ is trivial. {\color{dgreen} On such a region, given a metric $g$, we can choose a frame $(e_1,e_2)$ that is orthonormal with respect to $g$. Conversely, given a choice of frame $(e_1,e_2)$, there is a unique metric $g$ with respect to which the frame is orthonormal. In this context, we can encode any metric-compatible connection (not just the torsion-free Levi-Civita connection) with a (scalar-valued) one-form $A$ as follows.}

\begin{definition} \label{def:connection}
  Given a choice of frame $e_1,e_2$ for $T\region$ and a one-form $A$ on $\region$, we define a connection $\nabla$ via
  \begin{align}\label{eq:nablafromA}
    \nabla e_1&:=Ae_2,&\nabla e_2:=-Ae_1
  \end{align}
\end{definition}
These equations should be interpreted as $\nabla_Xe_1=A(X)e_2$ and $\nabla_Xe_2=-A(X)e_1$ for all vector fields $X$.

Given $(e_1,e_2)$ orthonormal with respect to a Riemannian metric $g$, there is a unique one-form $A$ that encodes the Levi-Civita connection via~(\ref{eq:nablafromA}).  For the moment, however, we will leave $A$ unspecified and study its evolution in generality.  In fact, we will temporarily dispense with $g$ and simply consider frames $e_1,e_2$ and one-forms $A$ that vary in time.  It will be helpful to also have a fixed frame $E_1,E_2$ that does not vary with time, and for simplicity we can set $E_i$ to be $e_i$ at $t=0$. We will also use this frame to define a reference connection $\tilde{\nabla}$.

\begin{definition}
  Let $E_1,E_2$ be a time-independent frame, specifically $E_i=e_i\rvert_{t=0}$. Let $u$ be the linear transformation (dependent on both space and time) that sends $E_i$ to $e_i$. Let $\tilde{\nabla}$ be the flat connection corresponding to the trivialization $E_1,E_2$, that is, $\tilde{\nabla} E_1=\tilde{\nabla} E_2=0$. As usual, $\tilde{\nabla}$ is extended to the tensor algebra via the Leibniz rule.
\end{definition}

Additionally, we define notation for ``$90^\circ$ counterclockwise rotation'' with respect to each of these frames.
\begin{definition}
  Let $J$ be the linear transformation, depending on both space and time, defined by $Je_1=e_2$ and $Je_2=-e_1$. Let $\tilde{J}$ be the linear transformation, depending on space but not time, such that $\tilde{J}E_1=E_2$ and $\tilde{J}E_2=-E_1$. Observe that $J=u\tilde{J}u^{-1}$.
\end{definition}

The difference between any two connections is a matrix-valued one-form.
\begin{definition}
  Let $a$ be the $\End(T\region)$-valued $1$-form $a:=\nabla-\tilde{\nabla}$. Equivalently, we can view $a$ as defining $\nabla$ via $\nabla:=\tilde{\nabla}+a$.
\end{definition}

\begin{notation}
  There can be confusion with $\End(T\region)$-valued $1$-forms, since, given a vector field, we can plug it into the one-forms or we can apply the linear transformations to it. We will follow convention and use the notation $\nabla_XY=\tilde{\nabla}_XY+a(X)Y$, the notation $\nabla_X=\tilde{\nabla}_X + a(X)$, and the notation $\nabla Y=\tilde{\nabla} Y+aY$. We think of $a(X)$ as a space-varying linear transformation, and we think of $aY$ as a vector-valued one-form.

  We will also encounter $\End(T\region)$-valued one-forms when we multiply a one-form and a section of $\End(T\region)$, such as $AJ$, which we interpret as $(AJ)(X)Y=(A(X))JY$. Another situation is the covariant derivative of a section of $\End(T\region)$, such as $\tilde{\nabla} u$, which we interpret as $(\tilde{\nabla} u)(X)Y=(\tilde{\nabla}_Xu)Y$. {\color{dgreen} We can also multiply an endomorphism-valued one-form $a$ by a linear transformation $u$, obtaining endomorphism-valued one-forms $au$ and $ua$, which we interpret as $(au)(X)Y=(a(X))uY$ and $(ua)(X)Y=u(a(X))Y$. We thus also have a commutator operator $[a, u] = au - ua$.}
\end{notation}

{\color{dgreen}
  We summarize the above definitions and notation in Table~\ref{tab:notation1}.
}

\begin{table}
  \color{dgreen}
  \centering
  \begin{tabular}{cp{4in}}
    \renewcommand{\arraystretch}{1.2}
    $E_1,E_2$&A frame that does not vary in time.\\
    $e_1,e_2$&A time-varying frame. In contexts with a metric $g$, this frame is orthonormal.\\
    $u$&A time-dependent linear transformation sending $(E_1,E_2)$ to $(e_1,e_2)$.\\
    $\xi$&$\dot uu^{-1}$.\\
    $\tilde J$&$90^\circ$ counterclockwise rotation with respect to the metric defined by $(E_1,E_2)$: $\tilde JE_1=E_2$, $\tilde JE_2=-E_1$.\\
    $J$&$90^\circ$ counterclockwise rotation with respect to the metric defined by $(e_1,e_2)$: $Je_1=e_2$, $Je_2=-e_1$.\\
    $\tilde\nabla$&The trivial connection defined by the frame $(E_1,E_2)$: $\tilde\nabla E_1=\tilde\nabla E_2=0$.\\
    $\nabla$&A connection compatible with the metric defined by $(e_1,e_2)$.\\
    $a$&The endomorphism-valued one-form of $\nabla$ with respect to the gauge $(E_1,E_2)$, defined by $\nabla_XE_i=a(X)E_i$.\\
    $AJ$&The endomorphism-valued one-form of $\nabla$ with respect to the gauge $(e_1,e_2)$, defined by $\nabla_Xe_i=A(X)Je_i$. $A$ itself is a scalar-valued one-form.\\
    $\mathcal Aa(X,Y)$&$a(X)Y-a(Y)X$.\\
    $\sigma$&$\frac\partial{\partial t}g$.\\                     
    $S\sigma$&$\sigma-g\Tr\sigma$.\\
    $\omega$&The volume form; $e^1\wedge e^2$.
  \end{tabular}
  \caption{Summary of notation in Section~\ref{sec:evolution}.}
  \label{tab:notation1}
\end{table}

Viewing $E_1$ and $E_2$ as given, we can now think about $\nabla$ as being defined in two different ways. The first way is via the one-form $A$ and the frame $(e_1,e_2)$, or, equivalently, via $A$ and the linear transformation $u$. The second way is via the matrix-valued one-form $a$. What is the relationship between $A$, $u$, and $a$?

\begin{proposition}
  We have
  \begin{equation}\label{eq:gauge}
    a=Au\tilde{J}u^{-1}-(\tilde{\nabla} u)u^{-1}=AJ-(\tilde{\nabla} u)u^{-1}.
  \end{equation}
\end{proposition}

\begin{proof}
  We rewrite \eqref{eq:nablafromA} as
  \begin{align*}
    \nabla(uE_1)&=Au\tilde{J}E_1,&\nabla(uE_2)&=Au\tilde{J}E_2.
  \end{align*}
  Meanwhile,
  \begin{align*}
    \tilde{\nabla}(uE_1)&=(\tilde{\nabla} u)E_1,&\tilde{\nabla}(uE_2)&=(\tilde{\nabla} u)E_2.
  \end{align*}
  Subtracting these equations, we obtain
  \begin{align*}
    auE_1&=(Au\tilde{J}-\tilde{\nabla} u)E_1,&auE_2=(Au\tilde{J}-\tilde{\nabla} u)E_2.
  \end{align*}
  The difference between two connections is tensorial, so we have
  \begin{equation*}
    au=Au\tilde{J}-\tilde{\nabla} u.
  \end{equation*}
  Multiplying both sides by $u^{-1}$ gives the desired result.
\end{proof}

One might also recognize \eqref{eq:gauge} as the equation for a gauge transformation. Indeed, $a$ is the matrix-valued $1$-form for $\nabla$ with respect to the trivialization $(E_1,E_2)$, whereas $AJ$ is the matrix-valued $1$-form for $\nabla$ with respect to the trivialization $(e_1,e_2)$.

Next, we discuss how \eqref{eq:gauge} changes with time.

\begin{proposition}\label{prop:adot}
  We have
  \begin{equation*}
    \dot a = \dot AJ-\nabla\xi,
  \end{equation*}
  where $\xi=\dot uu^{-1}$.
\end{proposition}

\begin{proof}
  Observe first that
  \begin{equation*}
    \dot J = \dot u\tilde{J}u^{-1}-u\tilde{J}u^{-1}\dot u u^{-1}=[\xi, J]=-[J, \xi].
  \end{equation*}
  Next, observe that
  \begin{equation*}
    \nabla\xi = \tilde{\nabla}\xi + [a,\xi].
  \end{equation*}
  Indeed, more generally, if $v$ is a section of $\End(T\region)$, and $X$ is a vector field, we have
  \begin{equation*}
    \nabla(vX) = (\nabla v)X+v\nabla X.
  \end{equation*}
  We likewise have that
  \begin{equation*}
    \tilde{\nabla}(vX) = (\tilde{\nabla} v)X+v\tilde{\nabla} X.
  \end{equation*}
  Subtracting the two equations, we have
  \begin{equation*}
    avX = (\nabla v-\tilde{\nabla} v)X + vaX,
  \end{equation*}
  so
  \begin{equation*}
    \nabla v = \tilde{\nabla} v + [a,v].
  \end{equation*}
  Continuing our computation of $\nabla\xi$, we focus on the first term and compute that
  \begin{align*}
    \tilde{\nabla}\xi&=(\tilde{\nabla}\dot u)u^{-1}-\dot uu^{-1}(\tilde{\nabla} u )u^{-1}\\
                   &=(\tilde{\nabla}\dot u)u^{-1}-(\tilde{\nabla} u)u^{-1}\dot u u^{-1}+ \left[(\tilde{\nabla} u)u^{-1},\dot uu^{-1}\right]\\
                   &=\dd{t}\left((\tilde{\nabla} u)u^{-1}\right)+\left[(\tilde{\nabla} u)u^{-1},\xi\right].
  \end{align*}

  We are now ready to differentiate \eqref{eq:gauge}. We compute
  \begin{align*}
    \dot a &= \dot AJ + A\dot J - \dd{t}\left((\tilde{\nabla} u)u^{-1}\right)\\
           &=\dot AJ - [AJ, \xi] - \tilde{\nabla}\xi+\left[(\tilde{\nabla} u)u^{-1},\xi\right]\\
           &=\dot AJ - \left(\tilde{\nabla}\xi + [a, \xi]\right)\\
           &=\dot AJ - \nabla\xi.\qedhere
  \end{align*}
\end{proof}

Here we place a warning that $\dot A$ is how the $1$-form changes, but if one would like to determine the evolution of its coefficients with respect to the basis $(e_1,e_2)$, one would need to take into account the fact that the basis is time-dependent, so there would be an additional term involving $\dot u$.

We now discuss the torsion-free condition. First, we need the following notation.
\begin{definition}
  We can think of $\End(T\region)$-valued one-forms in a different way by noting that they are sections of the bundle
  \begin{equation*}
    \Lambda^1(T^*\region)\otimes\End(T\region)=T^*\region\otimes T\region\otimes T^*\region.
  \end{equation*}
  Thus, there is a natural antisymmetrization map
  \begin{equation*}
    \A\colon\Lambda^1(T^*\region)\otimes\End(T\region)\to\Lambda^2(T^*\region)\otimes T\region,
  \end{equation*}
  defined by antisymmetrizing the two $T^*\region$ factors in $T^*\region\otimes T\region\otimes T^*\region$.
\end{definition}

The implication for the torsion-free condition is the following.
\begin{proposition}
  Assume that $\nabla$ is torsion-free at $t=0$. Then $\nabla$ remains torsion-free if and only if $\A\dot a=0$.
\end{proposition}

\begin{proof}
  The torsion-free condition is that
  \begin{equation*}
    \nabla_XY-\nabla_YX=[X,Y]
  \end{equation*}
  for all vector fields $X$ and $Y$. We can rewrite this equation as
  \begin{equation*}
    \tilde{\nabla}_XY-\tilde{\nabla}_YX+a(X)Y-a(Y)X=[X,Y].
  \end{equation*}
  In terms of our antisymmetrization operator, the above equation is
  \begin{equation}\label{eq:torsionfree}
    \tilde{\nabla}_XY-\tilde{\nabla}_YX+\A a(X,Y)=[X,Y].
  \end{equation}
  Note that, if $X$ and $Y$ are time-independent, then the only term in \eqref{eq:torsionfree} that depends on time is $\A a$. We assumed that $\nabla$ is torsion-free at $t=0$, so \eqref{eq:torsionfree} holds at $t=0$. Thus, it continues to hold at all future times if and only if $\A\dot a=0$.
\end{proof}

Combining with Proposition \ref{prop:adot}, we can then understand how $A$ evolves in time if $\nabla$ is torsion-free. We first need the following definitions and lemma.

\begin{definition}
  Let $\flat$ denote the time-dependent map $T\region\to T^*\region$ that sends $e_1\mapsto e^1$ and $e_2\mapsto e^2$, and let $\sharp$ denote its inverse. Let $\omega$ be the time-dependent $2$-form $e^1\wedge e^2$.
\end{definition}

\begin{lemma}\label{lemma:AAJ}
  If $\alpha$ is a one-form, then
  \begin{equation*}
    \A(\alpha J)=-\omega\,\alpha^\sharp
  \end{equation*}
\end{lemma}
\begin{proof}
  We compute
  \begin{align*}
    \A(\alpha J)&=\omega\,\A(\alpha J)(e_1,e_2)\\
                &=\omega\left(\alpha(e_1)Je_2-\alpha(e_2)Je_1\right)\\
                &=\omega\left(-\alpha(e_1)e_1-\alpha(e_2)e_2\right)\\
                &=-\omega\,\alpha^\sharp.\qedhere
  \end{align*}
\end{proof}

\begin{proposition}\label{prop:Adot}
  If $\nabla$ is torsion-free, then we have
  \begin{align*}
    \dot A&=\left(\nabla_{e_2}\xi e_1-\nabla_{e_1}\xi e_2\right)^\flat\\
    &=-\left((\A\nabla\xi)(e_1,e_2)\right)^\flat
  \end{align*}
  where $\xi=\dot uu^{-1}$ and the notation $\nabla_X\xi Y$ denotes first applying the covariant derivative to $\xi$, and then applying the resulting linear transformation to $Y$.
\end{proposition}

\begin{proof}
  Applying the antisymmetrization operator $\A$ to Proposition \ref{prop:adot} and using the torsion-free condition $\A\dot a=0$, we obtain that
  \begin{equation*}
    \A(\dot AJ)=\A(\nabla\xi).
  \end{equation*}
  By Lemma~\ref{lemma:AAJ}, the left-hand side is $-\omega\,\dot A^\sharp$. Moving on to the right-hand side, we have
  \begin{equation*}
    \A(\nabla\xi)=\omega\,\left(\nabla_{e_1}\xi e_2-\nabla_{e_2}\xi e_1\right).
  \end{equation*}
  Thus,
  \begin{equation*}
    (\dot A)^\sharp=(\nabla_{e_2}\xi)e_1-(\nabla_{e_1}\xi)e_2,
  \end{equation*}
  and the result follows.
\end{proof}

Let us now re-introduce the metric $g$ and interpret Proposition~\ref{prop:Adot} in the following context: We suppose that $g$ evolves with time and $e_1,e_2$ forms an orthonormal frame at all times.  Our setup ensures that the connection $\nabla$ will be the Levi-Civita connection in this context, since orthonormality of $e_1,e_2$ is tantamount to metric compatibility.  

We can relate the components of $\sigma = \frac{\partial}{\partial t}g$ to those of $u$ with the help of two observations.  
On one hand, we have
\begin{align*}
\sigma(e_1,e_1) &= 2\dot{e}^1(e_1) , \\
\sigma(e_2,e_2) &= 2\dot{e}^2(e_2), \\
\sigma(e_1,e_2) &= \dot{e}^1(e_2) + \dot{e}^2(e_1),
\end{align*}
since
\[
0 = \frac{d}{dt} e^i(e_j) = \dot{e}^i(e_j) + e^i(\dot{e}_j) = \dot{e}^i(e_j) + g(e_i,\dot{e}_j)
\]
and
\[
0 = \frac{d}{dt} g(e_i,e_j) = \sigma(e_i,e_j) + g(\dot{e}_i,e_j) + g(e_i,\dot{e}_j)
\]
for each $i,j=1,2$.  On the other hand, 
\begin{equation*}
  \dot e^i(e_j)=\dd{t}\left(e^i(e_j)\right)-e^i\left(\dot e_j\right)=-e^i\left(\dd{t}(uE_j)\right)=-e^i(\dot uE_j) =-e^i(\dot u u^{-1} e_j).
\end{equation*}
Thus, with $\xi=\dot{u}u^{-1}$, we have
\begin{align*}
  2\xi^1_1&=-\sigma_{11},&2\xi^1_2&=-\sigma_{12}-2f,\\
  2\xi^2_1&=-\sigma_{12}+2f,&2\xi^2_2&=-\sigma_{22},
\end{align*}
where $\xi^i_j = e^i(\xi e_j)$, $\sigma_{ij}=\sigma(e_i,e_j)$, and $f = \frac{1}{2}\sigma_{12} - \dot{e}^2(e_1) = -\frac{1}{2}\sigma_{12} + \dot{e}^1(e_2)$.
Put another way, we have a decomposition of $\xi$ into its symmetric and antisymmetric parts as
\begin{equation} \label{udotuinv}
  \xi = -\tfrac12\sigma^\sharp + fJ,
 \end{equation}
where, in this case, $\sharp$ refers to the map defined by $e^i\otimes e^j\mapsto e_i\otimes e^j$. Indeed, we note that, in coordinates, $J^1_2=e^1(Je_2)=e^1(-e_1)=-1$ and $J^2_1=e^2(Je_1) = e^2(e_2) =1$.

Observe that $\nabla J=0$. One way to see this is to observe that $J^\flat=\omega$, where $\flat$ is the inverse of $\sharp$. Metric compatibility gives $\nabla g=0$ and $\nabla\omega=0$. Moreover, note that $\flat$ is just contraction with $g$, so $(\nabla J)^\flat=\nabla\left(J^\flat\right)=\nabla\omega=0$. Since $\flat$ is an isomorphism, $\nabla J=0$. We thus have
\begin{equation*}
  \nabla\xi = -\tfrac12\nabla\sigma^\sharp+(df)J.
\end{equation*}
Antisymmetrizing both sides, we get
\[
\A\nabla\xi = -\tfrac12\A\nabla\sigma^\sharp+\A((df)J).
\]
Using Proposition~\ref{prop:Adot}, we have that 
\[
(\A\nabla\xi)(e_1,e_2) = -\dot{A}^\sharp.
\]
Meanwhile, using Lemma~\ref{lemma:AAJ}, we have,
\begin{equation*}
  \A((df)J)(e_1,e_2)=-(df)^\sharp.
\end{equation*}
We conclude that
\begin{equation*}
  \left(\dot A-df\right)^\sharp=\tfrac12\A\nabla\sigma^\sharp(e_1,e_2)=\tfrac12\left(\nabla_{e_1}\sigma^\sharp e_2-\nabla_{e_2}\sigma^\sharp e_1\right).
\end{equation*}
By metric compatibility, the $\sharp$ and $\flat$ maps commute with the covariant derivative, so, applying $\flat$ to the above equation (i.e. contracting with $g$), we obtain
\begin{equation}\label{eq:Adf}
  \dot A-df=\tfrac12\left(\nabla_{e_1}\sigma e_2-\nabla_{e_2}\sigma e_1\right),
\end{equation}
where we interpret $\nabla_{e_1}\sigma e_2$ as the one-form $X\mapsto(\nabla_{e_1}\sigma)(e_2,X)$.

It turns out that the right-hand side of the above equation is $-\frac{1}{2}$ of the Hodge star of $\dv S \sigma$.  We summarize and prove this fact below.
\begin{proposition} \label{prop:Adotdiv}
If $\nabla$ is the Levi-Civita connection associated with an evolving metric $g(t)$, and if $e_1,e_2$ is orthonormal at all times, then 
\begin{equation} \label{Adotdiv}
\dot{A} - df = -\frac{1}{2}\star \dv S \sigma,
\end{equation}
where $\sigma = \frac{\partial}{\partial t}g$, $\star$ denotes the Hodge star operator associated with $g$, and
\begin{equation} \label{gauge}
f = \frac{1}{2}\sigma_{12} - \dot{e}^2(e_1) = -\frac{1}{2}\sigma_{12} + \dot{e}^1(e_2).
\end{equation}
\end{proposition}
\begin{proof}
In light of \eqref{eq:Adf} and the fact that $\star\star\alpha=-\alpha$ for one-forms $\alpha$, it remains to check that $\dv S\sigma$ is the Hodge star of
\begin{equation*}
  \alpha = \nabla_{e_1}\sigma e_2-\nabla_{e_2}\sigma e_1.
\end{equation*}
Since $S\sigma = \sigma - g \Tr \sigma$, we have 
\begin{align*}
(\dv S \sigma)(e_2) 
&= \nabla_{e_1} (S\sigma)(e_1, e_2) + \nabla_{e_2} (S\sigma)(e_2, e_2) \\
&= (\nabla_{e_1} \sigma)(e_1,e_2) - \nabla_{e_1} (g \Tr \sigma)(e_1, e_2) + (\nabla_{e_2} \sigma)(e_2, e_2) - \nabla_{e_2} (g \Tr \sigma)(e_2, e_2) \\
&= (\nabla_{e_1} \sigma)(e_2,e_1) - g(e_2, e_1) \nabla_{e_1} \Tr \sigma + (\nabla_{e_2} \sigma)(e_2,e_2) - g(e_2,e_2) \nabla_{e_2} \Tr \sigma \\
&= (\nabla_{e_1} \sigma)(e_2,e_1) + (\nabla_{e_2} \sigma)(e_2,e_2) - \nabla_{e_2} \Tr \sigma.
\end{align*}
Since the trace commutes with covariant differentiation, 
\[
\nabla_{e_2} \Tr \sigma = \Tr \nabla_{e_2} \sigma = (\nabla_{e_2} \sigma)(e_1,e_1) + (\nabla_{e_2} \sigma)(e_2,e_2).
\]
Thus,
\begin{align*}
(\dv S \sigma)(e_2) 
&= (\nabla_{e_1} \sigma)(e_2,e_1) - (\nabla_{e_2} \sigma)(e_1,e_1) \\
&= \alpha(e_1).
\end{align*}
A similar calculation gives 
\[
(\dv S\sigma)(e_1) = -\alpha(e_2),
\]
so $\dv S\sigma$ is indeed the Hodge star of $\alpha$.
\end{proof}

We now remark on the relationship between the connection one-form $A$ and the Gaussian curvature $\kappa$.

{\color{dgreen}
\begin{proposition}
  The curvature of the connection $\nabla$ in Definition~\ref{def:connection} is the endomorphism-valued $2$-form $(dA)J$.
\end{proposition}
\begin{proof}
  Given vector fields $X$ and $Y$, we have
    \begin{equation*}
      \begin{split}
        &\phantom{={}}\nabla_X\nabla_Ye_1-\nabla_Y\nabla_Xe_1-\nabla_{[X,Y]}e_1\\
        &=\nabla_X(A(Y)e_2)-\nabla_Y(A(X)e_2)-A([X,Y])e_2\\
        &=\nabla_X(A(Y))e_2+A(Y)\nabla_Xe_2-\nabla_Y(A(X))e_2-A(X)\nabla_Ye_2-A([X,Y])e_2\\
        &=\left(\nabla_X(A(Y))-\nabla_Y(A(X))-A([X,Y])\right)e_2-A(Y)A(X)e_1+A(X)A(Y)e_1\\
        &=dA(X, Y)Je_1.
      \end{split}
    \end{equation*}
    The computation for $\nabla_X\nabla_Ye_2-\nabla_Y\nabla_Xe_2-\nabla_{[X,Y]}e_2$ is analogous.

    Alternatively, we can observe that, with respect to the gauge defined by the frame $(e_1,e_2)$, the endomorphism-valued one-form of the connection $\nabla$ is $AJ$. The curvature is $d(AJ)+\frac12[AJ\wedge AJ]$, which is equal to $(dA)J$; indeed, the matrix of $J$ with respect to the frame $(e_1,e_2)$ is constant, so $dJ=0$, and the commutator term vanishes because $J\in\mathfrak{so}(2)$, which is Abelian.
\end{proof}

\begin{remark}
The Gaussian curvature $\kappa$ is then $g(dA(e_1,e_2)Je_2,e_1)=-dA(e_1,e_2)$, so $\kappa\,\omega=-dA$. Thus, another way to obtain formula~(\ref{kappavoldot}) for $\frac{\partial}{\partial t}(\kappa \, \omega)$ is to take the exterior derivative of~(\ref{Adotdiv}), which yields
\begin{align*}
d\dot{A} = -\frac{1}{2} d \star \dv S \sigma = -\frac{1}{2} \star \star^{-1} d \star \dv S \sigma = -\frac{1}{2}\star \dv \dv S \sigma = -\frac{1}{2} (\dv \dv S \sigma) \, \omega
\end{align*}
since the operators $\star^{-1} d \star$ and $\dv$ coincide on one-forms.
\end{remark}
}

\subsection{Geodesic curvature evolution} \label{sec:klengthdot}

Next, we consider a curve $\mathcal{C}$ in $M$ and study the evolution of its geodesic curvature $k$, weighted by the induced length $1$-form on $\mathcal{C}$.  Let $\tau$ be a unit tangent and $n$ a unit normal, with the convention that $(n,\tau)$ is a right-handed frame, so for a circle oriented counterclockwise, $n$ is the outward normal. With this convention, we let the geodesic curvature be 
\[
k = -g(\nabla_\tau\tau, n) = g(\nabla_\tau n, \tau),
\] 
so the geodesic curvature of a counterclockwise circle is positive.  We let $ds$ be the natural length one-form on $\mathcal C$; that is, $ds=\tau^\flat$ with respect to the induced metric on $\mathcal C$. We now determine how the geodesic curvature evolves over time in terms of $\sigma$, or, more specifically, we compute the evolution of $k\,ds$ as a one-form on $\mathcal C$.

\begin{proposition} \label{prop:klengthdot}
  With notation as above,
  \begin{equation*}
    \frac{\partial}{\partial t}(k\,ds) = -\frac{1}{2}\bigl((\dv S\sigma)(n)+\nabla_\tau(\sigma(n,\tau))\bigr)\,ds.
  \end{equation*}
\end{proposition}

\begin{proof}
  The main idea is that we apply Proposition \ref{prop:Adotdiv} to the Frenet frame, $e_1=n$, $e_2=\tau$, restricting the one-forms in equation \eqref{Adotdiv} to $\mathcal C$, or, equivalently, evaluating these one-forms on $\tau$.

  We start by observing that applying \eqref{eq:nablafromA} to this frame implies
  \begin{equation*}
    A(\tau) = g(\nabla_\tau n, \tau) = k.
  \end{equation*}
  Letting $i\colon\mathcal C\to M$ denote the inclusion map, we therefore conclude that $i^*A = k\,ds$. The pullback of forms does not depend on the metric, so we can conclude that
  \begin{equation*}
    i^*\dot A = \frac{\partial}{\partial t}(k\,ds).
  \end{equation*}

  Next, we move on to $f$. We first observe that $\dot\tau$ is parallel to $\tau$, so $g(n, \dot\tau)=0$. In the notation of Section \ref{sec:conndot}, since the $e_i$ are an orthonormal frame, we can write this as $0=e^1(\dot e_2)=-\dot e^1(e_2)$, where the second equality comes from differentiating $e^1(e_2)=0$. We conclude then that $f=-\frac{1}{2}\sigma_{12}=-\frac{1}{2}\sigma(n,\tau)$. From here, we conclude that $df(\tau)=-\frac{1}{2}\nabla_\tau\bigl(\sigma(n,\tau)\bigr)$, and so
  \begin{equation*}
    i^*df=-\frac{1}{2}\nabla_\tau\bigl(\sigma(n, \tau)\bigr)\,ds.
  \end{equation*}

  Finally, we observe that since $\star$ rotates one-forms $90^\circ$ counterclockwise that $\bigl(\star\dv S\sigma\bigr)(\tau) = \bigl(\dv S\sigma\bigr)(n)$, so
  \begin{equation*}
    i^*\bigl(\star\dv S\sigma\bigr) = \bigl(\dv S\sigma\bigr)(n)\,ds.
  \end{equation*}

  Putting everything together, the restriction of \eqref{Adotdiv} to $\mathcal C$ tells us that
  \begin{equation*}
    \frac{\partial}{\partial t}(k\,ds)=-\frac{1}{2}\bigl(\nabla_\tau(\sigma(n, \tau))+(\dv S\sigma)(n)\bigr)\,ds.\qedhere
  \end{equation*}
\end{proof}

\subsection{Evolution of angles} \label{sec:angledot}

Next we study how angles evolve with time.  

{\color{blue}

\begin{proposition} \label{prop:frameangledot}
Let $(e_1(t),e_2(t))$ and $(\bar{e}_1(t),\bar{e}_2(t))$ be two time-varying frames that are each $g(t)$-orthonormal at all times.  Let $\theta(t)$ denote the counterclockwise angle by which $(e_1(t),e_2(t))$ is rotated relative to $(\bar{e}_1(t),\bar{e}_2(t))$, so that $\cos\theta = g(e_1,\overline{e}_1)$.  Then
\[
\frac{\partial}{\partial t}\theta = -\frac{1}{2}\sigma(e_1,e_2) + \frac{1}{2}\sigma(\bar{e}_1,\bar{e}_2) + \dot{e}^1(e_2) - \dot{\bar{e}}^1(\bar{e}_2),
\]
where $\sigma=\frac{\partial}{\partial t}g$.
\end{proposition}
\begin{proof}
Assume without loss of generality that $(\bar{e}_1(0),\bar{e}_2(0)) = (e_1(0),e_2(0))$.  
Let $u(t)$ denote the linear transformation that sends $(e_1(0),e_2(0))$ to $(e_1(t),e_2(t))$, and let $\bar{u}(t)$ denote the linear transformation that sends $(\bar{e}_1(0),\bar{e}_2(0))$ to $(\bar{e}_1(t),\bar{e}_2(t))$.  The matrix of $v(t)=u(t)\bar{u}(t)^{-1}$ with respect to the basis $(e_1(t),e_2(t))$ is a rotation by $\theta(t)$, so the matrix of $\dot{v}(t)v(t)^{-1}$ with respect to that basis is skew-symmetric with off-diagonal entries equal to $\pm \dot{\theta}$.  In other words,
\begin{align*}
\dot{\theta} 
&= g(\dot{v}v^{-1} e_1, e_2) \\
&= \frac{1}{2} \left( g(\dot{v}v^{-1} e_1, e_2)  - g(\dot{v}v^{-1} e_2, e_1) \right),
\end{align*}
where we used skew-symmetry to pass from the first to the second line.  Since
\begin{align*}
\dot{v}v^{-1} 
&= \dot{u} \bar{u}^{-1} \bar{u} u^{-1} - u \bar{u}^{-1} \dot{\bar{u}} \bar{u}^{-1} \bar{u} u^{-1} \\
&= \dot{u}u^{-1} - v \dot{\bar{u}} \bar{u}^{-1} v^{-1}
\end{align*}
and $v$ is an isometry, we find
\begin{align*}
\dot{\theta} 
&= \frac{1}{2} \left( g(\dot{u}u^{-1} e_1, e_2)  - g(\dot{u}u^{-1} e_2, e_1) \right) - \frac{1}{2} \left( g(v\dot{\bar{u}}\bar{u}^{-1} v^{-1} e_1, e_2)  - g(v\dot{\bar{u}}\bar{u}^{-1}v^{-1} e_2, e_1) \right) \\
&= \frac{1}{2} \left( g(\dot{u}u^{-1} e_1, e_2)  - g(\dot{u}u^{-1} e_2, e_1) \right) - \frac{1}{2} \left( g(\dot{\bar{u}}\bar{u}^{-1} v^{-1} e_1, v^{-1} e_2)  - g(\dot{\bar{u}}\bar{u}^{-1}v^{-1} e_2, v^{-1} e_1) \right) \\
&= \frac{1}{2} \left( g(\dot{u}u^{-1} e_1, e_2)  - g(\dot{u}u^{-1} e_2, e_1) \right) - \frac{1}{2} \left( g(\dot{\bar{u}}\bar{u}^{-1} \bar{e}_1, \bar{e}_2)  - g(\dot{\bar{u}}\bar{u}^{-1} \bar{e}_2, \bar{e}_1) \right).
\end{align*}
In view of~(\ref{udotuinv}) and the symmetry of $\sigma$, this is equal to
\[
\dot{\theta} = \frac{1}{2} \left( g( f J e_1, e_2) - g( f J e_2, e_1) \right) - \frac{1}{2} \left( g( \bar{f} \bar{J} \bar{e}_1, \bar{e}_2) - g( \bar{f} \bar{J} \bar{e}_2, \bar{e}_1) \right),
\]
where $J$ is the linear transformation that sends $(e_1,e_2)$ to $(e_2,-e_1)$, $\bar{J}$ is the linear transformation that sends $(\bar{e}_1,\bar{e}_2)$ to $(\bar{e}_2,-\bar{e}_1)$ (which equals $J$), and
\begin{align}
f &= \frac{1}{2}\sigma(e_1,e_2) -\dot{e}^2(e_1) = -\frac{1}{2}\sigma(e_1,e_2) + \dot{e}^1(e_2), \label{fdef} \\
\bar{f} &= \frac{1}{2}\sigma(\bar{e}_1,\bar{e}_2) - \dot{\bar{e}}^2(\bar{e}_1) = -\frac{1}{2}\sigma(\bar{e}_1,\bar{e}_2) + \dot{\bar{e}}^1(\bar{e}_2).
\end{align}
Simplifying, we get
\begin{align*}
\dot{\theta} 
&= f - \bar{f} \\
&= -\frac{1}{2}\sigma(e_1,e_2) + \frac{1}{2}\sigma(\bar{e}_1,\bar{e}_2) + \dot{e}^1(e_2) - \dot{\bar{e}}^1(\bar{e}_2).
\end{align*}
\end{proof}

}

\begin{figure}
\begin{center}
\begin{tikzpicture}[scale=1.3,dot/.style={fill,circle,inner sep=1.5pt}]
\path (80:3) coordinate (a) (20:3.5) coordinate (b) (0:0) node[dot,label=below left:$z$]{} coordinate (z) (-100:1) coordinate (e) (-160:1) coordinate (f) (80:2.0) coordinate (A) (20:1.5) coordinate (B) ($(z)!(A)!(B)$) (20:2.5) coordinate (Bt) (80:1.0) coordinate (At);
\draw[-]  (a) node [above left ]{$\mathcal{C}_1$} -- (e);
\draw[-]  (b) node [below right]{$\mathcal{C}_2$} -- (f);
\draw[very thick,->]  (A) -- node [below right]{$\tau_1$} (At);
\draw[very thick,->]  (B) -- node [above]{$\tau_2$} (Bt);
\tkzDefLine[orthogonal=through A](z,A)
\tkzDrawLine[very thick,add = -0.5 and -1.0,arrows=<-](A,tkzPointResult)
\tkzLabelLine[above](A,tkzPointResult){$n_1$}
\tkzDefLine[orthogonal=through B](z,B)
\tkzDrawLine[very thick,add = 0.65 and -1.0,arrows=<-](B,tkzPointResult)
\tkzLabelLine[below=2,right](B,tkzPointResult){$n_2$}
\path pic[draw, angle radius=9mm,"$\theta$",angle eccentricity=1.3] {angle = B--z--A};
\end{tikzpicture}
\end{center}
\caption{Configuration of the curves $\mathcal{C}_1$ and $\mathcal{C}_2$ in Proposition~\ref{prop:angledot}.}
\label{fig:curves}
\end{figure}
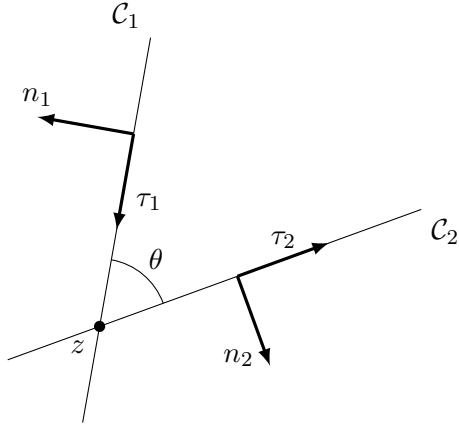

\begin{proposition} \label{prop:angledot} 
Let $\mathcal{C}_1$ and $\mathcal{C}_2$ be two curves in $M$ that intersect transversally at a point $z$.  If the metric $g$ evolves with time but $\mathcal{C}_1$ and $\mathcal{C}_2$ are fixed, then the angle $\theta$ between $\mathcal{C}_1$ and $\mathcal{C}_2$ at $z$ satisfies
\begin{equation} \label{angledot}
\frac{\partial}{\partial t}\theta = \frac{1}{2} \left( \sigma(\tau_2,n_2) - \sigma(\tau_1,n_1) \right),
\end{equation}
where $\sigma = \frac{\partial}{\partial t}g$ and $\tau_i,n_i$ are the unit tangent and unit normal vectors along $\mathcal{C}_i$.  Here, our convention is that $\cos\theta=-g(\tau_1,\tau_2)$, $g(n_1,\tau_2)<0$, and $(n_i,\tau_i)$ forms a right-handed frame for each $i$; see Figure~\ref{fig:curves}.
\end{proposition}
{\color{blue}
\begin{proof}
This is a special case of Proposition~\ref{prop:frameangledot}.  The angle $\theta$ between $\mathcal{C}_1$ and $\mathcal{C}_2$ at $z$ is the same as the angle between the frames $(\bar{e}_1,\bar{e}_2):=(n_2,\tau_2)$ and $(e_1,e_2) := (-n_1,-\tau_1)$ at $z$.  We have
\[
\dot{e}^1(e_2) = \frac{d}{dt}(e^1(e_2)) - e^1(\dot{e}_2) = -e^1(\dot{e}_2) = -g(e_1,\dot{e}_2) = -g(n_1,\dot{\tau}_1) =0
\]
since $\dot{\tau}_1$ is parallel to $\tau_1$.  Similarly, $\dot{\bar{e}}^1(\bar{e}_2)=0$, so~(\ref{angledot}) follows.
\end{proof}

\begin{proposition} \label{prop:dtheta} 
Let $(e_1,e_2)$ be a smoothly varying $g$-orthonormal frame field on a triangle $T$.  Along an edge $e \subset \partial T$ with outward unit normal vector $n$ and unit tangent vector $\tau$, let $\theta$ be the counterclockwise angle by which $(e_1,e_2)$ is rotated relative to $(n,\tau)$ at each point along $e$.  Then
\begin{equation} \label{thetadot}
\frac{\partial}{\partial t}\theta = \frac{1}{2}\sigma(n,\tau) + f
\end{equation}
and
\begin{equation} \label{dtheta}
d\theta(\tau) = A(\tau) - k,
\end{equation}
where $\sigma = \frac{\partial}{\partial t}g$, $f$ is given by~(\ref{gauge}), $A$ is the Levi-Civita connection one-form associated with $(e_1,e_2)$, and $k$ is the geodesic curvature of $e$. 
\end{proposition}
\begin{proof}
To compute the time rate of change of $\theta$, we apply Proposition~\ref{prop:frameangledot} with $(\bar{e}_1,\bar{e}_2)=(n,\tau)$.  As we saw in the proof of Proposition~\ref{prop:angledot}, we have $\dot{\bar{e}}^1(\bar{e}_2)=0$, so 
\begin{align*}
\dot{\theta} 
&= -\frac{1}{2}\sigma(e_1,e_2) + \frac{1}{2}\sigma(\bar{e}_1,\bar{e}_2) + \dot{e}^1(e_2) \\
&= \frac{1}{2} \sigma(n,\tau) + f.
\end{align*}

Next we compute $d\theta(\tau)$.  Differentiating the relation $\cos\theta = g(e_1,n)$ in the direction $\tau$ yields
\begin{align*}
-\sin\theta \, d\theta(\tau) 
&= \nabla_\tau \left( g(e_1,n) \right), \\
&= g(\nabla_\tau e_1, n) + g( e_1, \nabla_\tau n ) \\
&= g(A(\tau)e_2, n) + g(e_1, \nabla_\tau n) \\
&= -A(\tau)\sin\theta + g(e_1, \nabla_\tau n).
\end{align*}
Writing $e_1 = n \cos\theta + \tau \sin\theta$ and noting that $g(n,\nabla_\tau n) = \frac{1}{2} \nabla_\tau \left( g(n,n) \right) = 0$, we find that
\begin{align*}
-\sin\theta \, d\theta(\tau) 
&= -A(\tau)\sin\theta + g(\tau, \nabla_\tau n)\sin\theta \\
&= -A(\tau)\sin\theta + k\sin\theta.
\end{align*}
It follows that~(\ref{dtheta}) holds at all points on $e$ where $\sin\theta \neq 0$.  In the event that $\sin\theta=0$ on a subset $\widetilde{e} \subset e$ with positive length, we have $(e_1,e_2)=(\pm n, \pm\tau)$ and $A(\tau)=g(\nabla_{\tau} n, \tau) = k$ on $\widetilde{e}$, so~(\ref{dtheta}) holds on $\widetilde{e}$ with both sides of the equation equaling zero.  We conclude that~(\ref{dtheta}) holds everywhere on $e$ by continuity.
\end{proof}

}

\section{Distributional curvature} \label{sec:distcurv}

In this section, we give meaning to the distributional curvature of a Regge metric.

\paragraph{Notation.}

Let $\mathcal{S}$ be a triangulation of an oriented manifold of dimension $d=2$.  
Let $\mathcal{T}_h$, $\mathcal{E}_h$, and $\mathcal{V}_h$ denote the set of triangles, edges, and vertices, respectively, that comprise $\mathcal{S}$.  For the moment, the subscript $h$ serves no purpose other than to label discrete objects, but later it will be useful to consider families of triangulations parametrized by $h = \max_{T \in \mathcal{T}_h} \diam(T)$.  Let $\mathcal{V}_h^0 \subseteq \mathcal{V}_h$ denote the set of vertices $z \in \mathcal{V}_h$ that do not lie on $\partial \mathcal{S}$, and let $\mathcal{E}_h^0 \subseteq \mathcal{E}_h$ denote the set of interior edges---edges with at least one endpoint in $\mathcal{V}_h^0$.

On a triangle $T$, let $\mathcal{M}(T)$ denote the space of smooth Riemannian metrics on $T$.  {\color{blue}We think of elements of $\mathcal{M}(T)$ as symmetric $(0,2)$-tensor fields that are positive definite everywhere in $T$ and have smooth components.}   The space of Regge metrics is
\[
\mathcal{M} = \{ g \in \prod_{T \in \mathcal{T}_h} \mathcal{M}(T) \mid i_{T_1,e}^* g_{T_1} = i_{T_2,e}^* g_{T_2} \, \forall e = T_1 \cap T_2 \in \mathcal{E}_h^0  \},
\]
where $i_{T_j,e}^*$ denotes the pullback under the inclusion $i_{T_j,e} : e \xhookrightarrow{} T_j$.  In other words, the tangential-tangential component of a Regge metric $g$ along each interior edge $e$ is continuous.
For us, the most important Regge metrics are those that are piecewise polynomial.  When referring to the value of $g$ on a triangle $T \in \mathcal{T}_h$, we write either $g_T$ or (if there is no danger of confusion) simply $g$.  The Gaussian curvature $\kappa$ of $g_T$ is denoted
\[
\kappa_T(g) = \kappa(g_T).
\]

On each edge $e$ of $T$, we let $\tau$ and $n$ denote the unit tangent and unit normal with respect to $g_T$.  We assume that $n$ points outward and $(n,\tau)$ forms a right-handed frame.  The $L^2$-inner product of two scalar functions $u$ and $v$ on $e$ with respect to $g_T$ is denoted
\[
\langle u, v \rangle_{g,e} = \int_e u v \, \ds.
\]
Note that this integral does not depend on the triangle $T \supset e$ under consideration, since the tangential-tangential component of $g$ is the same on both sides of $e$.  
The geodesic curvature of $e$ is
\[
k_e(g_T) = -g_T(n,\nabla_\tau \tau).
\]
In general, when $e \in \mathcal{E}_h^0$ lies on the boundary of two triangles $T_1$, $T_2$, the geodesic curvature of $e$ measured by $g_{T_1}$ need not agree with the geodesic curvature of $e$ measured by $g_{T_2}$.  For such an edge $e$, we denote
\[
\llbracket k_e(g) \rrbracket = k_e( g_{T_1} ) + k_e( g_{T_2} ).
\]
We use similar notation for the jumps in other quantities across edges.  Thus, for example, if $v$ is a function whose normal derivatives along $e$ are well-defined, then we denote the jump in $\nabla_n v$ across $e$ by
\[
\left\llbracket \nabla_n v \right\rrbracket = \left. (\nabla_n v) \right|_{T_1} + \left. (\nabla_n v) \right|_{T_2}.
\]
If $e \in \mathcal{E}_h \setminus \mathcal{E}_h^0$, then we set
\[
\left\llbracket \nabla_n v \right\rrbracket = \nabla_n v
\]
along $e$.

At a vertex $z \in \mathcal{V}_h^0$, the angle defect at $z$ is
\[
\Theta_z(g) = 2\pi - \sum_{T \in \, \Omega_z} \theta_{zT},
\]
where $\, \Omega_z$ is the set of triangles in $\mathcal{T}_h$ sharing the vertex $z$, and $\theta_{zT}$ is the interior angle of $T$ at $z$ as measured by $g_T$.

\paragraph{Curvature.}
We wish to give meaning to the distributional curvature two-form associated with a Regge metric.  To do so, we introduce a space of functions on which this distribution will act.  Although the space of continuous functions would suffice, we prefer to use a Sobolev space to facilitate the analysis in Section~\ref{sec:conv}.

On a triangle $T$, we let $W^{k,p}(T)$ denote the Sobolev space of differentiability index $k \in \mathbb{N}_0$ and integrability index $p \in [1,\infty]$.  We write $H^k(T) = W^{k,2}(T)$.  
For functions $v \in H^k(T)$ with $k \ge 1$, the trace of $v$ on any edge $e \subset \partial T$ is well-defined.  We denote
\[
V = \left\{ v \in \prod_{T \in \mathcal{T}_h} H^2(T) \mid \left.v_{T_1}\right|_e  = \left.v_{T_2}\right|_e \, \forall e = T_1 \cap T_2 \in \mathcal{E}_h^0, \, \left.v\right|_{\partial \mathcal{S}} = 0 \right\}.
\]
Note that the Sobolev embedding theorem guarantees that elements of $H^2(T)$ are continuous, so $v(z)$ is well-defined and single-valued at each vertex $z \in \mathcal{V}_h$ when $v \in V$.  We use the notation $V'$ to refer to the dual space of $V$.

We are now ready to define the distributional curvature two-form associated with $g$.
\begin{definition} \label{def:distcurv}
Let $g$ be a Regge metric.  The \emph{distributional curvature two-form} associated with $g$ is the linear functional $\distcurv(g) \in V'$ defined by
\begin{equation} \label{distcurv}
\langle \distcurv(g), v \rangle_{V',V} = \sum_{T \in \mathcal{T}_h} \langle \kappa_T(g), v \rangle_{g,T} + \sum_{e \in \mathcal{E}_h^0} \langle \llbracket k_e(g) \rrbracket, v \rangle_{g,e} + \sum_{z \in \mathcal{V}_h^0} \Theta_z(g) v(z)
\end{equation}
for every $v \in V$.
\end{definition}

To motivate this definition, notice that when $g$ is a piecewise constant Regge metric, only the third term in~(\ref{distcurv}) is present, so we recover the standard definition of the curvature two-form on triangulations: a summation of Dirac delta distributions supported at the vertices $z \in \mathcal{V}_h^0$, each weighted by the angle defect at $z$.  When $g$ varies within each triangle $T$, we encounter additional contributions from the first two terms in~(\ref{distcurv}).  The role of the first term is self-evident.  To understand the term involving $\llbracket k_e(g) \rrbracket$, consider a thin four-sided region $U$ enclosing a portion $\widetilde{e}$ of an edge $e \in \mathcal{E}_h^0$, two of whose sides consist of points having geodesic distance $\epsilon/2$ from $\widetilde{e}$, and two of whose sides are (non-smooth) geodesics of length $\epsilon$ that intersect $e$ orthogonally with respect to $g$, as in Figure~\ref{fig:edge}.  If $g$ were smooth, then the Gauss-Bonnet theorem would yield $\int_U \kappa \, \omega + \int_{\partial U} k \, ds = 2\pi$.  For small $\epsilon$, we have $\int_{\partial U} k \, ds \approx -\int_{\widetilde{e}} \llbracket k_e(g) \rrbracket \, ds + 2\pi$, where the second term comes from summing the angles of the four corners of $U$.  Hence, $\int_U \kappa \, \omega \approx \int_{\widetilde{e}} \llbracket k_e(g) \rrbracket \, ds$, which suggests that the curvature two-form should include a Dirac delta distribution supported on $e$ and weighted by $\llbracket k_e(g) \rrbracket$ whenever $\llbracket k_e(g) \rrbracket$ is nonzero.  In Section~\ref{sec:linearization} we give a more systematic justification of Definition~\ref{def:distcurv} by showing that the linearization of $\distcurv(g)$ around a given Regge metric $g$ is precisely the linearized curvature operator, interpreted in a distributional sense.  {\color{blue}See also~\cite{strichartz2020defining} for another justification of Definition~\ref{def:distcurv}.}

\begin{figure}
\centering
\begin{tikzpicture}[
  triangle/.style={
   regular polygon,
   regular polygon sides=3,
   minimum size=3cm,
   draw}
]
\node [triangle,rotate=30] (a) {};
\node [triangle,rotate=210,left=1.5cm of a] (b) {};
\draw[draw=black,fill=gray,fill opacity=0.3] (-0.9,-0.75) rectangle ++(0.3,1.5);
\end{tikzpicture}
\caption{Region containing a portion of an edge $e \in \mathcal{E}_h^0$.}
\label{fig:edge}
\end{figure}
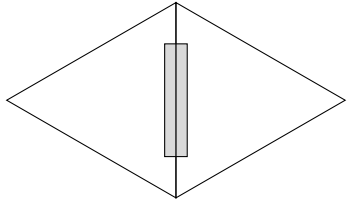

\section{Linearization of the distributional curvature} \label{sec:linearization}

Our next goal is to determine how the distributional curvature two-form evolves under deformations of the metric.

We will prove the following result.
\begin{theorem} \label{thm:main}
Let $g(t)$ be a Regge metric depending smoothly on $t$.  Then
\[
\frac{d}{dt} \langle \distcurv(g), v \rangle_{V',V} = \frac{1}{2} b_h(g; \sigma, v)
\]
for every $v \in V$, where $\sigma = \frac{\partial}{\partial t} g$ and
\begin{equation} \label{bh}
b_h(g;\sigma,v) = \sum_{T \in \mathcal{T}_h} \langle S\sigma, \Hess v \rangle_{g,T} + \sum_{e \in \mathcal{E}_h} \left\langle \sigma(\tau,\tau), \left\llbracket \nabla_n v \right\rrbracket \right\rangle_{g,e}.
\end{equation}
\end{theorem}

The theorem above has important ramifications, since it allows one to analyze the distributional curvature $\distcurv(g)$ by analyzing the $g$-dependent bilinear form $b_h(g; \cdot,\cdot)$.  This bilinear form is well-studied in the finite element literature: it is a non-Euclidean generalization of the bilinear form used to discretize the Euclidean $\dv\dv$ operator in the Hellan-Herrmann-Johnson finite element method.  Hence, classical techniques from finite element theory can be brought to bear.  We use this observation in Section~\ref{sec:conv} to analyze the convergence of $\distcurv(g)$ under refinement.

In our study of $b_h(g;\sigma,v)$ below, we will primarily deal with the setting in which $g$ is a time-dependent Regge metric with time derivative $\sigma$.  In particular, $\sigma$ will be piecewise smooth.  But for future reference, we remark that for any Regge metric $g$, $b_h(g; \cdot, \cdot)$ extends to a bounded bilinear form on $\Sigma \times V$, where
\[
\Sigma = \left\{ \sigma \in \prod_{T \in \mathcal{T}_h} H^1 S_2^0(T) \mid i_{T_1,e}^* \sigma_{T_1} = i_{T_2,e}^* \sigma_{T_2} \, \forall e = T_1 \cap T_2 \in \mathcal{E}_h^0 \right\}.
\]
Here, $H^k S_2^0(T)$ denotes the space of all symmetric $(0,2)$-tensor fields on $T$ with coefficients in $H^k(T)$.

\subsection{Proof of Theorem~\ref{thm:main}}

Let us prove Theorem~\ref{thm:main}.
We first use Propositions~\ref{prop:kappavoldot},~\ref{prop:klengthdot}, and~\ref{prop:angledot} to differentiate the three terms on the right-hand side of~(\ref{distcurv}). We get
\begin{align}
\frac{d}{dt} \langle \kappa_T(g),v \rangle_{g,T} &= \frac{d}{dt} \int_T v \kappa_T(g) \, \omega = \frac{1}{2} \int_T  v (\dv \dv S \sigma) \, \omega = \frac{1}{2} \langle \dv \dv S \sigma, v \rangle_{g,T}, \label{term1dot}
\end{align}
\begin{align}
\frac{d}{dt} \langle \llbracket k_e(g) \rrbracket, v \rangle_{g,e} = \frac{d}{dt} \int_e v  \llbracket  k_e(g) \rrbracket \ds &= -\frac{1}{2} \int_e v \left\llbracket (\dv S \sigma)(n) + \nabla_\tau (\sigma(n,\tau)) \right\rrbracket \ds \nonumber \\
&= -\frac{1}{2} \left\langle \left\llbracket (\dv S \sigma)(n) + \nabla_\tau (\sigma(n,\tau)) \right\rrbracket, v \right\rangle_{g,e}, \label{term2dot}
\end{align}
and
\begin{equation}
\frac{d}{dt} \Theta_z(g) v(z) = - v(z) \frac{d}{dt} \sum_{T \in \Omega_z} \theta_{zT} = \frac{1}{2} v(z) \sum_{T \in \Omega_z} \llbracket \sigma(n,\tau) \rrbracket_{zT}, \label{term3dot}
\end{equation}
where
\begin{equation} \label{jumpvertex}
\llbracket \sigma(n,\tau) \rrbracket_{zT} = \left. \sigma(n,\tau) \right|_{e^{(1)}}(z) - \left. \sigma(n,\tau) \right|_{e^{(2)}}(z)
\end{equation}
and $e^{(1)},e^{(2)} \subset \partial T$ are the two edges of $T$ having an endpoint at $z$.  Here, our convention is that $e^{(1)}$ and $e^{(2)}$ are ordered so that the unit tangent vector along $e^{(1)}$ points toward $z$ and the unit tangent vector along $e^{(2)}$  points away from $z$.

Next we show that these are precisely the quantities that appear if one integrates $b_h$ by parts.

\begin{proposition} \label{prop:bhrewritten}
For any $(g,\sigma,v) \in \mathcal{M} \times \Sigma \times V$ satisfying $\left.\sigma\right|_T \in H^2 S_2^0(T)$ for every $T \in \mathcal{T}_h$, we have
\begin{equation} \label{bhrewritten}
\begin{split}
b_h(g;\sigma,v) &= \sum_{T \in \mathcal{T}_h} \langle \dv \dv S \sigma, v \rangle_{g,T} \\&\;\;\; - \sum_{e \in \mathcal{E}_h^0} \left\langle \left\llbracket (\dv S \sigma)(n) + \nabla_\tau (\sigma(n,\tau)) \right\rrbracket, v \right\rangle_{g,e} \\&\;\;\;
+ \sum_{z \in \mathcal{V}_h^0} v(z) \sum_{T \in \, \Omega_z} \llbracket \sigma(n,\tau) \rrbracket_{zT}.
\end{split}
\end{equation}
\end{proposition} 
\begin{remark}
The Euclidean version of this identity appears in~\cite[Equation 2.4]{comodi1989hellan} and~\cite[Lemma 2.1]{chen2021finite}.
\end{remark}
\begin{proof}[Proof {\color{blue}of Proposition~\ref{prop:bhrewritten}}]
We start by rewriting the integrals over $e \in \mathcal{E}_h$ in~(\ref{bh}) as integrals over triangle boundaries:
\begin{equation} \label{bhrewrite}
b_h(g;\sigma,v) = \sum_{T \in \mathcal{T}_h} \left( \int_T \langle S\sigma, \Hess v \rangle \, \omega + \int_{\partial T} \sigma(\tau,\tau) \nabla_n v \, \ds \right).
\end{equation}
Next, by thinking of $\nabla v$ as a vector field and using $S\sigma\nabla v$ to denote the one-form $X \mapsto S\sigma(\nabla v, X)$, we use the identity
\[
\dv(S\sigma \nabla v) = (\dv S \sigma)(\nabla v) + \langle S \sigma, \Hess v \rangle
\]
and the divergence theorem to write
\begin{align}
\int_T \langle S\sigma, \Hess v \rangle \, \omega 
&= \int_T \dv(S\sigma \nabla v) \, \omega - \int_T (\dv S \sigma)(\nabla v) \, \omega \nonumber \\
&= \int_{\partial T} S\sigma(n, \nabla v) \, \ds  - \int_T (\dv S \sigma)(\nabla v) \, \omega. \label{SHess}
\end{align}
Since $\nabla v = \tau \nabla_\tau v + n \nabla_n v$, we have
\begin{align}
\int_{\partial T} S\sigma(n, \nabla v) \, \ds
&= \int_{\partial T} S\sigma(n,\tau) \nabla_\tau v \, \ds + \int_{\partial T} S\sigma(n,n) \nabla_n v \, \ds \nonumber \\
&= \int_{\partial T} \sigma(n,\tau) \nabla_\tau v \, \ds - \int_{\partial T} \sigma(\tau,\tau) \nabla_n v \, \ds. \label{Ssigma}
\end{align}
In the second line above, we used the identities
\begin{align*}
S\sigma(n,\tau) &= \sigma(n,\tau) - (\Tr \sigma) g(n,\tau) = \sigma(n,\tau), \\
S\sigma(n,n) &= \sigma(n,n) - (\Tr \sigma)g(n,n) = \sigma(n,n) - (\sigma(\tau,\tau)+\sigma(n,n)) = -\sigma(\tau,\tau).
\end{align*}
It follows from~(\ref{bhrewrite}),~(\ref{SHess}), and~(\ref{Ssigma}) that
\begin{equation} \label{bhgradv}
b_h(g; \sigma, v) = \sum_{T \in \mathcal{T}_h} \left( -\int_T (\dv S \sigma)(\nabla v) \, \omega + \int_{\partial T} \sigma(n,\tau) \nabla_\tau v \, \ds \right).
\end{equation}
Next, we use the integration by parts formula
\begin{align*}
\int_T (\dv S \sigma)(\nabla v) \, \omega
&= \int_{\partial T} (\dv S \sigma)(n) v  \, \ds - \int_T (\dv \dv S \sigma) v \, \omega
\end{align*}
to write
\[
b_h(g;\sigma,v) = \sum_{T \in \mathcal{T}_h} \bigg( \int_T (\dv \dv S \sigma) v \, \omega - \int_{\partial T} (\dv S \sigma)(n) v  \, \ds +  \int_{\partial T} \sigma(n,\tau) \nabla_\tau v \, \ds \bigg).
\]
On each triangle $T$, the last term above involves integrals over the three edges of $T$, each of which can be integrated by parts to give
\begin{align*}
\int_e \sigma(n,\tau) \nabla_\tau v \, \ds 
&= \left. \sigma(n,\tau)v \right|_{z^{(1)}}^{z^{(2)}} - \int_e \nabla_\tau (\sigma(n,\tau)) v \, \ds.
\end{align*}
Here, $z^{(1)}$ and $z^{(2)}$ are the endpoints of $e$, ordered so that $\tau$ points from $z^{(1)}$ to $z^{(2)}$.
It follows that
\[\begin{split}
b_h(g;\sigma,v) = \sum_{T \in \mathcal{T}_h} \int_T (\dv \dv S \sigma) v \, \omega - \sum_{e \in \mathcal{E}_h} \int_e \left\llbracket (\dv S \sigma)(n) + \nabla_\tau (\sigma(n,\tau)) \right\rrbracket v \, \ds \\
+ \sum_{z \in \mathcal{V}_h} v(z) \sum_{T \in \, \Omega_z} \llbracket \sigma(n,\tau) \rrbracket_{zT}.
\end{split}\]
Finally, since $v$ vanishes on $\partial \mathcal{S}$, the sums over $e \in \mathcal{E}_h$ and $z \in \mathcal{V}_h$ can be replaced by sums over $e \in \mathcal{E}_h^0$ and $z \in \mathcal{V}_h^0$.  
This completes the proof of~(\ref{bhrewritten}).
\end{proof}

Theorem~\ref{thm:main} follows from comparing~(\ref{term1dot}-\ref{term3dot}) with~(\ref{bhrewritten}).

\begin{table}
  \centering
  \color{dgreen}
  \begin{tabular}{cp{4.3in}}
    \renewcommand{\arraystretch}{1.2}
    $\mathcal M$&Regge metrics: metrics which are smooth on each triangle and have tangential-tangential interelement continuity.\\
    $V$&Scalar functions with interelement continuity and vanishing on the boundary:\newline$\left\{ v \in \prod_{T \in \mathcal{T}_h} H^2(T) \mid \left.v_{T_1}\right|_e  = \left.v_{T_2}\right|_e \, \forall e = T_1 \cap T_2 \in \mathcal{E}_h^0, \, \left.v\right|_{\partial \mathcal{S}} = 0 \right\}$.\\
    $\Sigma$&Symmetric $(0,2)$-tensor fields with tangential-tangential interelement continuity:\newline$\left\{ \sigma \in \prod_{T \in \mathcal{T}_h} H^1 S_2^0(T) \mid i_{T_1,e}^* \sigma_{T_1} = i_{T_2,e}^* \sigma_{T_2} \, \forall e = T_1 \cap T_2 \in \mathcal{E}_h^0 \right\}$.\\
    $W$&One-forms with tangential interelement continuity and tangential vanishing on the boundary:\newline{$\begin{aligned}\textstyle\Bigl\{ \alpha \in \prod_{T \in \mathcal{T}_h} H^1\Lambda^1(T) \mid &i_{T_1,e}^* \alpha_{T_1} = i_{T_2,e}^* \alpha_{T_2} \, \forall e = T_1 \cap T_2 \in \mathcal{E}_h^0,\\[-.5\baselineskip] &\qquad i_{T,e}^* \alpha_T = 0 \,\forall e \in \mathcal{E}_h \setminus \mathcal{E}_h^0, T \supset e \qquad \Bigr\}.\end{aligned}$}\\
    $X$&Two-forms {\color{blue}with vanishing mean:\newline$X = \left\{ F \in \prod_{T \in \mathcal{T}_h} L^2\Lambda^2(T) \mid \sum_{T \in \mathcal{T}_h} \int_T F = 0 \right\}$.}\\
    $U$&Vector fields: $\{u \in H^1(\Omega) \otimes \mathbb{R}^2 \mid \left.u\right|_T \in  H^2(T) \otimes \mathbb{R}^2, \, \forall T \in \mathcal{T}_h \}$. \\
  \end{tabular}
  \caption{Summary of notation in Sections~\ref{sec:distcurv}--\ref{sec:poly}.}
  \label{tab:notation2}
\end{table}

\section{Distributional connection} \label{sec:distconn}

In this section, we discuss how to associate with a Regge metric $g$ a distributional connection one-form: a one-form that encodes the Levi-Civita connection in the sense of Definition~\ref{def:connection}.  The one-form we construct will have the property that its distributional exterior \emph{co}derivative is equal to $-\distcurv(g)$.  As such, it can be thought of as a distributional version of the Hodge star of $A$ from Definition~\ref{def:connection}.

Our construction will make use of a space of differential one-forms having single-valued (tangential) trace on element interfaces and vanishing (tangential) trace on $\partial \mathcal{S}$.  Let $H^1\Lambda^1(T)$ denote the space of one-forms on a triangle $T$ with coefficients in $H^1(T)$, and let
\[\begin{split}
W = \{ \alpha \in \prod_{T \in \mathcal{T}_h} H^1\Lambda^1(T) \mid i_{T_1,e}^* \alpha_{T_1} = i_{T_2,e}^* \alpha_{T_2} \, \forall e = T_1 \cap T_2 \in \mathcal{E}_h^0, \\ i_{T,e}^* \alpha_T = 0 \, \forall e \in \mathcal{E}_h \setminus \mathcal{E}_h^0, T \supset e  \}.
\end{split}\]
We also set
\[
{\color{blue}
X = \left\{ F \in \prod_{T \in \mathcal{T}_h} L^2\Lambda^2(T) \mid \sum_{T \in \mathcal{T}_h} \int_T F = 0 \right\},
}
\]
where $L^2\Lambda^2(T)$ denotes the space of square-integrable two-forms on $T$.  Note that $dV \subset W$ and $dW \subset X$.  In addition, on contractible domains, the sequence
\begin{center}
  \begin{tikzcd}
    0 \arrow{r}{} & V \arrow{r}{d} & W \arrow{r}{d} & X \arrow{r}{} & 0
  \end{tikzcd}
\end{center}
is exact, and hence so its dual; see Appendix~\ref{sec:appendix}.

{\color{blue}

Let $g$ be a Regge metric, and let $(e_1,e_2)$ be a $g$-orthonormal frame field that is smooth on each triangle $T \in \mathcal{T}_h$.  We do not assume that $(e_1,e_2)$ enjoys any interelement continuity.  On a triangle $T \in \mathcal{T}_h$, we let $A_T(g)$ denote the Levi-Civita connection one-form associated with $(e_1,e_2)$, so that (in the notation of Section~\ref{sec:conndot}) $\nabla e_1 = A_T(g)e_2$ and $\nabla e_2 = -A_T(g)e_1$ on $T$.

On an oriented edge $e \in \mathcal{E}_h^0$ shared by two triangles $T^+$ and $T^-$, the frame $(e_1,e_2)$ generally differs on either side of $e$.  We let $(e_1^+,e_2^+)$ and $(e_1^-,e_2^-)$ denote the values of $(e_1,e_2)$ on the two sides of $e$, with the convention that $e$ is oriented positively with respect to $T^+$ and negatively with respect to $T^-$; see Figure~\ref{fig:frames}.  We also let $n^+$ and $n^-$ denote the unit normal vectors on the two sides of $e$, with the convention that $n^+$ points outward from $T^+$ and has unit length with respect to $g_{T^+}$, and $n^-$ points outward from $T^-$ and has unit length with respect to $g_{T^-}$.  We define unit tangent vectors $\tau^+$ and $\tau^-$ by requiring that $(n^+,\tau^+)$ and $(n^-,\tau^-)$ form orthonormal frames with respect to $g_{T^+}$ and $g_{T^-}$, respectively.  By the tangential-tangential continuity of $g$, $\tau^+=-\tau^-$.  We will sometimes abbreviate $\tau^+$ as $\tau$, thereby recovering our notation from previous sections.

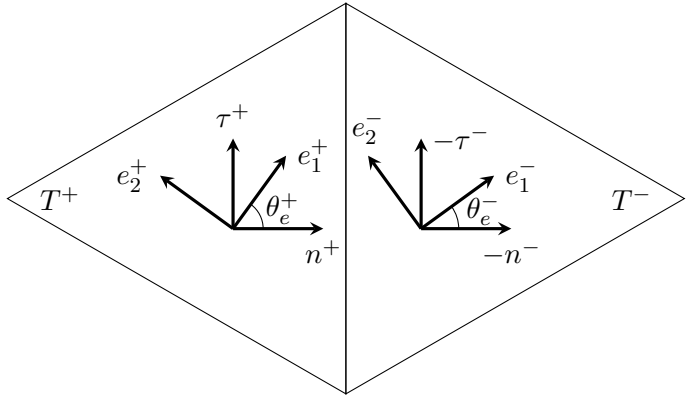
\begin{figure}
\begin{center}
\begin{tikzpicture}[
  triangle/.style={
   regular polygon,
   regular polygon sides=3,
   minimum size=6cm,
   draw}
]
\node [triangle,rotate=30] (a) {};
\node [triangle,rotate=210,left=3cm of a] (b) {};

\node at (-5.3,0) {$T^+$};
\node at (2.3,0) {$T^-$};

\coordinate (op) at (-3,-0.4) {};
\coordinate (np) at (-1.8,-0.4) {};
\coordinate (tp) at (-3,0.8) {};
\coordinate (e1p) at (-2.29466, 0.57082) {};
\coordinate (e2p) at (-3.97082, 0.305342) {};
\draw[very thick,-stealth](op)--(np) node[below]{$n^+$};
\draw[very thick,-stealth](op)--(tp) node[above]{$\tau^+$};
\draw[very thick,-stealth](op)--(e1p) node[right]{$e_1^+$};
\draw[very thick,-stealth](op)--(e2p) node[left]{$e_2^+$};
\path pic[draw, angle radius=4mm,"$\quad\theta_e^+$",angle eccentricity=1.3,pic text options={shift={(0pt,1pt)}}] {angle = np--op--e1p};

\coordinate (om) at (-0.5,-0.4) {};
\coordinate (nm) at (0.7,-0.4) {};
\coordinate (tm) at (-0.5,0.8) {};
\coordinate (e1m) at (0.47082, 0.305342) {};
\coordinate (e2m) at (-1.20534, 0.57082) {};
\draw[very thick,-stealth](om)--(nm) node[below]{$-n^-$};
\draw[very thick,-stealth](om)--(tm) node[right]{$-\tau^-$};
\draw[very thick,-stealth](om)--(e1m) node[right]{$e_1^-$};
\draw[very thick,-stealth](om)--(e2m) node[above]{$e_2^-$};
\path pic[draw, angle radius=5mm,"$\quad\theta_e^-$",angle eccentricity=1.4,pic text options={shift={(-1pt,1pt)}}] {angle = nm--om--e1m};
\end{tikzpicture}
\end{center}
\caption{\color{blue}Frames on either side of an edge shared by two triangles $T^+$ and $T^-$.  In this example, the metrics $g_{T^+}$ and $g_{T^-}$ are Euclidean, so $\tau^+=-\tau^-$ and $n^+ = -n^-$.  In general, $n^+$ need not equal $-n^-$.}
\label{fig:frames}
\end{figure}

We use $\theta_e^+$ to denote the counterclockwise angle by which $(e_1^+,e_2^+)$ is rotated relative to $(n^+,\tau^+)$.  Likewise, $\theta_e^-$ denotes that counterclockwise angle by which $(e_1^-,e_2^-)$ is rotated relative to $(-n^-,-\tau^-)$.  The difference
\[
\theta_e = \theta_e^+ - \theta_e^-
\]
will play an important role in what follows.  When we wish to emphasize its dependence on $g$ and/or $(e_1,e_2)$, we write $\theta_e(g)$ or $\theta_e(g,e_1,e_2)$.  This angle has the following interpretation.  Consider the linear transformation $\Pi_{T^-T^+} : \operatorname{span}\{e_1^+,e_2^+\} \to  \operatorname{span}\{e_1^-,e_2^-\}$ that sends $(n^+,\tau^+)$ to $(-n^-,-\tau^-)$.  Because $\Pi_{T^-T^+}$ is an isometry, the matrix of $\Pi_{T^-T^+}$ with respect to the bases $(e_1^+,e_2^+)$ and $(e_1^-,e_2^-)$ is an orthogonal matrix.  In fact, its value is readily seen to be
\[
\begin{pmatrix}
\cos\theta_e & -\sin\theta_e \\
\sin\theta_e & \cos\theta_e
\end{pmatrix}.
\]
Thus, thinking of $\Pi_{T^-T^+}$ as a parallel transport operator across the edge $e$, we can think of $\theta_e$ as the angle by which a vector rotates, relative to the (discontinuous) frame $(e_1,e_2)$, when it is parallel transported from $T^+$ to $T^-$.

{\color{dgreen}
  With this in mind, we are almost ready to define the distributional connection one-form associated with the frame. However, there is one complication that we must address first: the angle $\theta_e$ is a priori only defined up to a multiple of $2\pi$. This issue is intimately connected to the fact that, in the smooth setting, one can only define the connection one-form on a region that has a continuous frame. Such a frame always exists locally, but there can be topological obstructions to defining it globally.

  {\color{blue}In} our setting, we have a discontinuous frame. Nonetheless, across an edge, it is always possible to smooth the frame, and the angle $\theta_e$ tells us how. For instance, in Figure~\ref{fig:frames}, if we take $\theta_e$ to be $30^\circ$, then we quickly rotate the frame slightly clockwise as we cross the edge, but if we take $\theta_e$ to be the equivalent angle $\theta_e=-330^\circ$, then we quickly rotate the frame almost all the way around counterclockwise as we cross the edge. Thus, even though $30^\circ$ and $-330^\circ$ are equivalent angles, the choice impacts the topology of how we might smooth the frame. Moreover, depending on the choices we make, due to topological obstructions it may or may not be possible to smooth the frame near a vertex. To ensure that we have ``topologically consistent'' choices of $\theta_e$, we do the following.

  First, we assume that our {\color{blue}triangulation} $\mathcal{S}$ admits a globally constant (but not necessarily $g$-orthonormal) frame field $(\widehat{e}_1,\widehat{e}_2)$.  By this we mean that $(\widehat{e}_1,\widehat{e}_2)$ is constant on every triangle $T \in \mathcal{T}_h$, and $\widehat{g}(\widehat{e}_1,\tau)$ and $\widehat{g}(\widehat{e}_2,\tau)$ are single-valued on every edge $e \in \mathcal{E}_h^0$, where $\widehat{g}$ is the unique metric with respect to which $(\widehat{e}_1,\widehat{e}_2)$ is $\widehat{g}$-orthonormal. As a result, $\theta^+_e(\widehat{g}, \widehat{e}_1,\widehat{e_2})$ and $\theta^-_e(\widehat{g}, \widehat{e}_1,\widehat{e_2})$ represent the same angle, so, a priori, $\theta_e(\widehat{g}, \widehat{e}_1,\widehat{e}_2)$ is a multiple of $2\pi$. We now require that, for each edge, we make a consistent choice of angle representative, so that $\theta_e^+(\widehat{g}, \widehat{e}_1,\widehat{e_2})=\theta_e^-(\widehat{g}, \widehat{e}_1,\widehat{e_2})$, and hence $\theta_e(\widehat{g}, \widehat{e}_1,\widehat{e}_2)=0$.

  Next, on every triangle, we continuously vary the frame $(\widehat{e}_1,\widehat{e}_2)$ to the frame $(e_1,e_2)$. Consequently, the metric defined by this frame varies from $\widehat{g}$ to $g$. The angles $\theta_e^\pm$ thus also vary continuously, thereby giving us a specific value of $\theta_e^{\pm}(g, e_1, e_2)$ that is defined without the previous ambiguity up to multiples of $2\pi$. We remark that $\theta_e^{\pm}$ will depend on exactly how we vary the frame from $(\widehat{e}_1,\widehat{e}_2)$ to the frame $(e_1,e_2)$, but the important thing is that we will be ``topologically consistent'' around vertices in a way that will become evident in the proof of Proposition~\ref{prop:ddistconn}.

  We do remark that one systematic way of varying the frame from $(\widehat e_1,\widehat e_2)$ to $(e_1,e_2)$ is as follows. We can linearly interpolate between the metric $\widehat{g}$ and the metric $g$. Consequently, we can use equation~\eqref{udotuinv} with $\sigma=\dot g$ and $f=0$ to obtain $\xi=\dot uu^{-1}$. Solving this differential equation for the linear transformation $u$, we obtain a frame $(u\widehat e_1, u\widehat e_2)$ that is orthonormal at all times. In particular, at the final time, the frame is orthonormal with respect to $g$; call this $g$-orthonormal frame $(\bar e_1,\bar e_2)$. Our remaining task is to continuously rotate $(\bar e_1,\bar e_2)$ until it agrees with $(e_1,e_2)$. The angle $\psi$ between these frames is continuous, but the difficulty is, once again, that $\psi$ is a priori only defined up to $2\pi$; if at one point we continuously rotate the frame by $30^\circ$ and at a nearby point we continuously rotate the frame by $-330^\circ$, the result will be discontinuous. However, since on each triangle $T$, both frames are continuous, and $T$ is contractible, we can always choose a continuous real-valued (as opposed to angle-valued) $\psi$ and use this choice to rotate $(\bar e_1,\bar e_2)$ to $(e_1,e_2)$ in a continuous manner on $T$.

  Once we have chosen the particular representative $\theta_e$ for the angle between frames on adjacent triangles as discussed above, we can define the distributional connection one-form.
}

\begin{definition} \label{def:distconn}
Let $g$ be a Regge metric.  
Let $(e_1,e_2)$ be a $g$-orthonormal frame field that is smooth on each triangle $T \in \mathcal{T}_h$.
The \emph{distributional connection one-form} associated with $(e_1,e_2)$ is the linear functional $\Gamma_{\rm dist}(g) \in W'$ defined by
\[
\langle \Gamma_{\rm dist}(g), \alpha \rangle_{W',W} = \sum_{T \in \mathcal{T}_h} \langle \star A_T(g), \alpha \rangle_{g,T} - \sum_{e \in \mathcal{E}_h^0} \langle \theta_e(g), \alpha(\tau) \rangle_{g,e}, \quad \forall \alpha \in W.
\]
\end{definition}

{\color{dgreen}
\begin{remark}
  We remark that, formally, the distributional connection one-form is associated not just to the frame $(e_1,e_2)$ but to the specific path from $(\widehat{e}_1,\widehat{e}_2)$ to $(e_1,e_2)$; different choices of path will lead to $\theta_e$ that differ by a multiple of $2\pi$. However, we note that our results hold regardless of which path we choose.
\end{remark}
}

\begin{remark}
The parallel transport operator $\Pi_{T^-T^+}$ that we used to motivate Definition~\ref{def:distconn} is a widely used notion of parallel transport on piecewise flat triangulations~\cite{leok2005discrete,crane2010trivial,li2018regge,christiansen2011linearization}.    It is also used implicitly in~\cite[Chapter 3]{li2018regge} when studying geodesics on triangulations equipped with general (not necessarily piecewise flat) Regge metrics.
\end{remark}

Next we will show that the distributional exterior coderivative of $\Gamma_{\rm dist}(g)$ is ${\color{dgreen}-\distcurv(g)}$.

\begin{proposition} \label{prop:ddistconn}
The (Hodge star of the) distributional exterior coderivative of $\Gamma_{\rm dist}(g)$ is $-\distcurv(g)$.  That is,
\[
\langle \Gamma_{\rm dist}(g), dv \rangle_{W',W} = -\langle \distcurv(g), v \rangle_{V',V}, \quad \forall v \in V.
\]
\end{proposition}
\begin{proof}
For any $v \in V$, we have
\begin{equation} \label{Gammadv}
\langle \Gamma_{\rm dist}(g), dv \rangle_{W',W} = \sum_{T \in \mathcal{T}_h} \langle \star A_T(g), dv \rangle_{g,T} - \sum_{e \in \mathcal{E}_h^0} \langle \theta_e(g), dv(\tau) \rangle_{g,e}.
\end{equation}
On each triangle $T$, we can use Stokes' theorem to write
\begin{align}
\langle \star A_T(g), dv \rangle_{g,T} 
&= \int_T A_T(g) \wedge dv \nonumber \\
&= -\int_{\partial T} A_T(g) v + \int_T dA_T(g) v \nonumber \\
&= -\int_{\partial T} A_T(g) v  - \int_T \kappa_T(g) v \, \omega, \label{triint}
\end{align}
where $\kappa_T(g)$ is the Gaussian curvature of $g$ within $T$.  
On any edge $e$ shared by two triangles $T^+$ and $T^-$, we have
\begin{align*}
\langle \theta_e(g), dv(\tau) \rangle_{g,e}
&= \int_e \theta_e^+(g) dv(\tau) \, ds - \int_e \theta_e^-(g) dv(\tau) \, ds. 
\end{align*}
We can integrate each of these terms by parts and use Proposition~\ref{prop:dtheta} to compute
\begin{align}
\int_e \theta_e^{\pm}(g) dv(\tau) \, ds
&= \theta_e^{\pm}(g) v \big|_{z^{(1)}}^{z^{(2)}} - \int_e d\theta_e^{\pm}(g)(\tau) v \, ds \nonumber \\
&= \theta_e^{\pm}(g) v \big|_{z^{(1)}}^{z^{(2)}} \mp \int_e A_{T^{\pm}}(g)(\tau) v \, ds \pm \int_e k_e(g_{T^{\pm}}) v \, ds \nonumber \\
&= \theta_e^{\pm}(g) v \big|_{z^{(1)}}^{z^{(2)}} \mp \int_e A_{T^{\pm}}(g) v \pm \int_e k_e(g_{T^{\pm}}) v \, ds, \label{edgeint}
\end{align}
where $z^{(1)},z^{(2)} \in \mathcal{V}_h$ are the two endpoints of $e$, and we have taken care to note the different sign conventions for $\theta_e^+$ and $\theta_e^-$ when invoking Proposition~\ref{prop:dtheta}.  When we substitute these relations into~(\ref{Gammadv}), the integrals of $A_{T^{\pm}}(g)v$ over edges in~(\ref{edgeint}) cancel with the integrals of $A_{T}(g)v$ over triangle boundaries in~(\ref{triint}).  Noting that $v$ vanishes on $\partial \mathcal{S}$, we are left with
\begin{align*}
\langle \Gamma_{\rm dist}(g), dv \rangle_{W',W} 
&= -\sum_{T \in \mathcal{T}_h} \int_T \kappa_T(g) v \, \omega - \sum_{e \in \mathcal{E}_h^0} \int_e \llbracket k_e(g) \rrbracket v \, ds - \sum_{z \in \mathcal{V}_h^0} \widetilde{\Theta}_z(g) v(z)  \\
&= -\sum_{T \in \mathcal{T}_h} \langle \kappa_T(g), v \rangle_{g,T} - \sum_{e \in \mathcal{E}_h^0} \langle \llbracket k_e(g) \rrbracket, v \rangle_{g,e} - \sum_{z \in \mathcal{V}_h^0} \widetilde{\Theta}_z(g) v(z),
\end{align*}
where
\begin{equation} \label{angledefecttilde}
\widetilde{\Theta}_z(g) = \sum_{e \supset z} s_{ez} \theta_e(g)(z)
\end{equation}
and $s_{ez}=+1$ if $e$ points toward $z$ and $s_{ez}=-1$ if $e$ points away from $z$.
It remains to show that $\widetilde{\Theta}_z(g)$ equals the angle defect $\Theta_z(g)$, which we recall is given by
\[
\Theta_z(g) = 2\pi - \sum_{T \in \Omega_z} \theta_{zT}.
\]  
To show that $\widetilde{\Theta}_z(g)=\Theta_z(g)$, consider a vertex $z$ shared by $m$ triangles $T_0,T_1,T_2,\dots,T_m=T_0$.  Assume that these triangles are ordered so that for each $i=1,2,\dots,m$, $T_{i-1}$ and $T_i$ share an edge $e_{i-1/2}$ with one endpoint at $z$.  Assume that each such edge $e_{i-1/2}$ points toward $z$, so that the numbers $s_{ez}$ in~(\ref{angledefecttilde}) are all $+1$ for this vertex $z$.  Consider for each $i$ the linear transformation $p_{T_i}$ that rotates vectors in $T_i$ at $z$ counterclockwise by $\theta_{zT_i}$ radians with respect to the metric $g_{T_i}$ assigned to $T_i$ at $z$.  The matrix of $p_{T_i}$ with respect to the basis $\left(e_1\big|_{T_i}(z),e_2\big|_{T_i}(z)\right)$ is
\[
P_{T_i} := \begin{pmatrix} \cos \theta_{zT_i} & -\sin \theta_{zT_i} \\ \sin \theta_{zT_i} & \cos \theta_{zT_i} \end{pmatrix},
\]  
and the matrix of $\Pi_{T_i T_{i-1}}$ at $z$ with respect to the bases $\left(e_1\big|_{T_{i-1}}(z),e_2\big|_{T_{i-1}}(z)\right)$ and $\left(e_1\big|_{T_i}(z),e_2\big|_{T_i}(z)\right)$ is
\[
R_{T_i T_{i-1}} := \begin{pmatrix}
\cos\theta_{e_{i-1/2}} & -\sin\theta_{e_{i-1/2}}  \\
\sin\theta_{e_{i-1/2}} & \cos\theta_{e_{i-1/2}} 
\end{pmatrix}.
\]
The product
\[
q = p_{T_0} \Pi_{T_0 T_{m-1}} p_{T_{m-1}} \Pi_{T_{m-1} T_{m-2}} \cdots p_{T_2} \Pi_{T_2 T_1} p_{T_1} \Pi_{T_1 T_0}
\]
{\color{dgreen} is the identity operator: it is a rotation, and one can check that $p_{T_i}\Pi_{T_iT_{i-1}}$ sends the tangent vector of $e_{i-1/2}$ to the tangent vector of $e_{i+1/2}$.}
Thus,
\begin{align*}
\begin{pmatrix} 1 & 0 \\ 0 & 1 \end{pmatrix} 
&= P_{T_0} R_{T_0 T_{m-1}} P_{T_{m-1}} R_{T_{m-1} T_{m-2}} \cdots P_{T_2} R_{T_2 T_1} P_{T_1} R_{T_1 T_0} \\
&= \begin{pmatrix} \cos\left(\widetilde{\Theta}_z(g) + \sum_{i=1}^m \theta_{zT_i}\right) & -\sin\left(\widetilde{\Theta}_z(g) + \sum_{i=1}^m \theta_{zT_i}\right) \\ \sin\left(\widetilde{\Theta}_z(g) + \sum_{i=1}^m \theta_{zT_i}\right) & \cos\left(\widetilde{\Theta}_z(g) + \sum_{i=1}^m \theta_{zT_i}\right) \end{pmatrix},
\end{align*}
where $\widetilde{\Theta}_z(g)$ is given by~(\ref{angledefecttilde}).
It follows that $\widetilde{\Theta}_z(g) + \sum_{i=1}^m \theta_{zT_i} \in 2\pi\mathbb{Z}$.  One can see that in fact $\widetilde{\Theta}_z(g) + \sum_{i=1}^m \theta_{zT_i} = 2\pi$ by continuously deforming $g$ to a flat metric on $\cup_{i=1}^m T_i$ (as discussed above {\color{dgreen}Definition~\ref{def:distconn})}, in which case we clearly have $\sum_{i=1}^m \theta_{zT_i}=2\pi$ and $\widetilde{\Theta}_z=0$.  It follows that 
\[
\widetilde{\Theta}_z(g) =  2\pi - \sum_{i=1}^m \theta_{zT_i} = \Theta_z(g).
\]
\end{proof}

The following proposition explains how the distributional connection one-form behaves under a change in frame.

\begin{proposition} \label{prop:framechange}
{\color{dgreen} Let $g$ be a Regge metric, and let $(\bar{e}_1,\bar{e}_2)$ be a piecewise smooth $g$-orthonormal frame field with distributional connection one-form $\bar{\Gamma}_{\rm dist}(g)$. Let $\psi$ be a piecewise smooth scalar field (not necessarily continuous), let $(e_1,e_2)$ be the piecewise smooth $g$-orthonormal frame field obtained by rotating $(\bar{e}_1, \bar{e}_2)$ counterclockwise by $\psi$, and let $\Gamma_{\rm dist}(g)$ be the corresponding distributional connection one-form.} We have
\[
\langle \Gamma_{\rm dist}(g) - \bar{\Gamma}_{\rm dist}(g), \alpha \rangle_{W',W} = -\langle \psi, d\alpha \rangle_{X',X}, \quad \forall \alpha \in W,
\]
where $\langle \psi, F \rangle_{X',X} := \sum_{T \in \mathcal{T}_h} \int_T \psi F$ for all $F \in X$.
\end{proposition}
\begin{proof}
Using bars to denote quantities that are computed with respect to $(\bar{e}_1,\bar{e}_2)$, we have 
\[
\theta_e^{\pm}(g) = \bar{\theta}_e^{\pm}(g) + \psi^{\pm}
\]
on each edge $e \in \mathcal{E}_h^0$, where $\psi^+$ and $\psi^-$ denote the values of $\psi$ on opposite sides of $e$.  Meanwhile, the change-of-frame formula for smooth connection one-forms yields
\[
A_{T}(g) = \bar{A}_T(g) + d\psi.
\]
It follows that 
\begin{align*}
\langle \Gamma_{\rm dist}(g) - \bar{\Gamma}_{\rm dist}(g), \alpha \rangle_{W',W} 
&= \sum_{T \in \mathcal{T}_h} \langle \star d\psi, \alpha \rangle_{g,T} - \sum_{e \in \mathcal{E}_h^0} \langle \llbracket \psi \rrbracket, \alpha(\tau) \rangle \\
&= \sum_{T \in \mathcal{T}_h} \left( \int_T d\psi \wedge \alpha - \int_{\partial T} \psi \alpha \right) \\
&= -\sum_{T \in \mathcal{T}_h} \int_T \psi d\alpha.
\end{align*}
\end{proof}

Next we study how the distributional connection one-form evolves under deformations of the metric.

\begin{proposition} \label{prop:distconndot}
Let $g(t)$ be a Regge metric depending smoothly on $t$.   Let $(e_1(t),e_2(t))$ be a $g(t)$-orthonormal frame field that is smooth on each triangle $T \in \mathcal{T}_h$ and depends smoothly on $t$.  Then
\[
\frac{d}{dt} \langle \Gamma_{\rm dist}(g(t)), \alpha \rangle_{W',W} = -\frac{1}{2} c_h(g(t); \sigma(t), \alpha) - \frac{d}{dt} \langle F(t), d\alpha \rangle_{X',X}, \quad \forall \alpha \in W,
\]
where $\sigma = \frac{\partial}{\partial t}g$.  In this formula, $c_h : \mathcal{M} \times \Sigma \times W \rightarrow \mathbb{R}$ is given by
\begin{align*}
c_h(g; \sigma, \alpha) 
&= \sum_{T \in \mathcal{T}_h} \langle S \sigma, \nabla \alpha \rangle_{g,T} + \sum_{e \in \mathcal{E}_h} \langle \sigma(\tau,\tau), \llbracket \alpha(n) \rrbracket \rangle_{g,e},
\end{align*}
and $F(t) \in X'$ is given by
\[
\langle F(t), G \rangle_{X',X} = \sum_{T \in \mathcal{T}_h} \int_T \widetilde{f}(t) G, \quad \forall G \in X,
\]
where $\widetilde{f}(t) = \int_0^t f(t')\,dt'$, and $f = \frac{1}{2}\sigma(e_1,e_2) - \dot{e}^2(e_1) = -\frac{1}{2}\sigma(e_1,e_2) + \dot{e}^1(e_2)$.
\end{proposition}
}
\begin{remark}
We have
\begin{equation} \label{chdv}
c_h(g;\sigma,dv) = b_h(g; \sigma, v), \quad \forall (g,\sigma,v) \in \mathcal{M} \times \Sigma \times V,
\end{equation}
since $\nabla dv = \Hess v$ and $dv(n) = \nabla_n v$.  {\color{blue}Thus, taking $\alpha=dv$ in Proposition~\ref{prop:distconndot} provides an alternative proof of Theorem~\ref{thm:main}.}
\end{remark}
\begin{remark} \label{remark:ch2}
An equivalent definition of $c_h$, which can be derived using integration by parts, is
\begin{equation} \label{ch2}
c_h(g; \sigma, \alpha) = -\sum_{T \in \mathcal{T}_h} \langle \dv S \sigma, \alpha \rangle_{g,T} + \sum_{e \in \mathcal{E}_h^0} \langle \llbracket \sigma(n,\tau) \rrbracket, \alpha(\tau) \rangle_{g,e}.
\end{equation}
Here, $\llbracket \sigma(n,\tau) \rrbracket$ denotes the jump in $\sigma(n,\tau)$ across the edge $e$, which is not to be confused with the quantity $\llbracket \sigma(n,\tau) \rrbracket_{zT}$ defined in~(\ref{jumpvertex}).
\end{remark}
{\color{blue}
\begin{proof}[Proof of Proposition~\ref{prop:distconndot}]
For any $\alpha \in W$, we have
\[
\frac{d}{dt}\langle \Gamma_{\rm dist}(g(t)), \alpha \rangle_{W',W} = \sum_{T \in \mathcal{T}_h} \frac{d}{dt} \langle \star A_T(g(t)), \alpha \rangle_{g,T} - \sum_{e \in \mathcal{E}_h^0} \frac{d}{dt} \langle \theta_e(g(t)), \alpha(\tau) \rangle_{g,e}.
\]
On each triangle $T$, Proposition~\ref{prop:Adotdiv} implies
\begin{align*}
\frac{d}{dt} \langle \star A_T(g(t)), \alpha \rangle_{g,T} 
&= \frac{d}{dt} \int_T A_T(g(t)) \wedge \alpha \\
&= -\frac{1}{2} \int_T \star \dv S \sigma \wedge \alpha + \int_T df \wedge \alpha  \\
&= -\frac{1}{2} \int_T  \star \dv S \sigma \wedge \alpha + \int_{\partial T} f \alpha - \int_T f d\alpha  \\
&= \frac{1}{2} \langle \dv S \sigma, \alpha \rangle_{g,T} + \int_{\partial T} f \alpha - \langle \star f, d\alpha \rangle_{g,T},
\end{align*}
where $f = \frac{1}{2}\sigma(e_1,e_2) - \dot{e}^2(e_1) = -\frac{1}{2}\sigma(e_1,e_2) + \dot{e}^1(e_2)$.  On each interior edge $e$, Proposition~\ref{prop:dtheta} implies
\begin{align*}
\frac{d}{dt} \langle \theta_e(g(t)), \alpha(\tau) \rangle_{g,e}
&= \frac{d}{dt} \int_e \theta_e(g(t)) \alpha(\tau) \, ds \\
&= \frac{d}{dt} \int_e \left( \theta_e^+(g(t)) - \theta_e^-(g(t)) \right) \alpha \\
&= \int_e \left\llbracket \frac{1}{2}\sigma(n,\tau) + f \right\rrbracket \alpha \\
&= \frac{1}{2} \langle \llbracket \sigma(n,\tau) \rrbracket, \alpha(\tau) \rangle_{g,e} + \int_e \llbracket f \rrbracket \alpha.
\end{align*}
When we sum over all triangles $T$ and all interior edges $e$, the integrals of $f\alpha$ over triangle boundaries cancel with the integrals of $\llbracket f \rrbracket \alpha$ over edges.  Noting that $\alpha(\tau)$ vanishes on $\partial \mathcal{S}$, we get 
\[
\frac{d}{dt}\langle \Gamma_{\rm dist}(g(t)), \alpha \rangle_{W',W} = \frac{1}{2} \sum_{T \in \mathcal{T}_h} \langle \dv S \sigma, \alpha \rangle_{g,T} - \frac{1}{2} \sum_{e \in \mathcal{E}_h^0} \langle \llbracket \sigma(n,\tau) \rrbracket, \alpha(\tau) \rangle_{g,e} - \sum_{T \in \mathcal{T}_h} \langle \star f, d\alpha \rangle_{g,T}.
\]
In view of Remark~\ref{remark:ch2} and the fact that
\[
\frac{d}{dt} \langle F(t), d\alpha \rangle_{X',X} = \frac{d}{dt} \sum_{T \in \mathcal{T}_h} \int_T \widetilde{f}(t) d\alpha =   \sum_{T \in \mathcal{T}_h} \int_T f(t) d\alpha =  \sum_{T \in \mathcal{T}_h} \langle \star f(t), d\alpha \rangle_{g,T},
\]
we see that
\[
\frac{d}{dt}\langle \Gamma_{\rm dist}(g(t)), \alpha \rangle_{W',W} = -\frac{1}{2} c_h(g(t); \sigma(t), \alpha) - \frac{d}{dt} \langle F(t), d\alpha \rangle_{X',X}.
\]
\end{proof}

\paragraph{Choosing a frame field.}  Definition~\ref{def:distconn} assumes that one has selected a piecewise smooth $g$-orthonormal frame field $(e_1,e_2)$ in advance.  On a triangulated planar domain, there is a natural way to construct such a frame field.  Starting from a frame field $E_1,E_2$ that is orthonormal and globally parallel with respect to the Euclidean metric $\delta$, we deform both the metric and the frame field until the frame field is $g$-orthonormal.  The deforming metric can be taken to be $(1-t)\delta + tg$ with $t \in [0,1]$.  According to~(\ref{udotuinv}), the frame field must then satisfy $e_1(t)=u(t)E_1$ and $e_2(t)=u(t)E_2$, where $u$ is the linear transformation (dependent on both space and time) that satisfies the differential equation $\dot{u}(t)u(t)^{-1}=-\frac{1}{2}(g-\delta)^{\sharp} + fJ$ with initial condition $u(0)=\operatorname{id}$ and $f$ arbitrary.  Here, the sharp is taken with respect to the metric $(1-t)\delta+tg$.  Choosing $f=0$ for simplicity, we obtain (up to an overall rotation) a canonical $g$-orthonormal frame field $(e_1,e_2)$ at time $t=1$.  We refer to the distributional connection one-form associated with this frame field as the canonical distributional connection one-form.  With the help of Proposition~\ref{prop:distconndot}, one can compute it directly without ever constructing $(e_1,e_2)$.

\begin{definition} \label{def:distconn2}
Let $g$ be a Regge metric on a triangulated planar domain.  The \emph{canonical distributional connection one-form} is the distributional connection one-form $\Gamma_{\rm dist}(g)$ associated with the canonical frame field $(e_1,e_2)$ constructed above.  Equivalently,
\begin{equation} \label{distconn2}
\langle \Gamma_{\rm dist}(g), \alpha \rangle_{W',W} = -\frac{1}{2} \int_0^1 c_h((1-t)\delta + tg; g-\delta, \alpha) \, dt, \quad \forall \alpha \in W.
\end{equation}
\end{definition}

}

\section{Convergence} \label{sec:conv}

In this section, we study the convergence of the distributional curvature and distributional connection under refinement.  We restrict our attention to the setting in which $\mathcal{S}$ triangulates a polygonal domain $\Omega \subset \mathbb{R}^2$.  We suppose that $\mathcal{S}$ belongs to a family of such triangulations parametrized by
\[
h = \max_{T \in \mathcal{T}_h} h_T,
\]
where $h_T = \diam T$ denotes the diameter of a triangle $T$.  We assume this family is shape-regular.  That is, there exists a constant $C$ independent of $h$ such that
\[
\max_{T \in \mathcal{T}_h} \frac{h_T}{\rho_T} \le C
\]
for every $h>0$, where $\rho_T$ denotes the inradius of $T$.  We reuse the letter $C$ below to denote a constant independent of $h$ which is not necessarily the same at each occurrence.

For $\alpha \in W$ and $v \in V$, we denote
\[
\|\alpha\|_{\color{blue}W,h} = \|\alpha\|_{L^2(\Omega)} + \left( \sum_{T \in \mathcal{T}_h} h_T^2 |\alpha|_{H^1(T)}^2 \right)^{1/2}
\]
and
\[
\|v\|_{\color{blue}V,h} = \|dv\|_{\color{blue}W,h} =  |v|_{H^1(\Omega)} + \left( \sum_{T \in \mathcal{T}_h} h_T^2 |v|_{H^2(T)}^2 \right)^{1/2},
\]
where $|\cdot|_{H^k(T)}$ denotes the $H^k(T)$-seminorm. 
The dual norms are denoted
\begin{align*}
\|\beta\|_{\color{blue}W',h} &= 
\sup_{\alpha \in W} \frac{\langle \beta, \alpha \rangle_{W',W}}{\|\alpha\|_{\color{blue}W,h}}, \\
\|u\|_{\color{blue}V',h} &= 
\sup_{v \in V} \frac{\langle u, v \rangle_{V',V}}{\|v\|_{\color{blue}V,h}}.
\end{align*}
Note that $\|\cdot\|_{\color{blue}V,h}$ is a valid norm on $V$ by the Poincar\'e inequality and the containment $V \subset H^1_0(\Omega)$.

Given a smooth Riemannian metric $g$ on $\Omega$, we identify its curvature two-form $\kappa(g) \omega(g)$ with an element of $V'$ via
\[
\langle \kappa(g) \omega(g), v \rangle_{V',V} := \int_\Omega v \kappa(g) \omega(g).
\]
In order to discuss convergence of the canonical distributional connection~(\ref{distconn2}), we also need to single out a canonical smooth connection one-form $A(g)$.  In view of~(\ref{Adotdiv}) and~(\ref{distconn2}), we select
\[
A(g) = -\frac{1}{2} \int_0^1 \star \dv S (g-\delta) \, dt,
\]
where the operators $\star$, $\dv$, and $S$ in the integrand are taken with respect to the metric $G(t) = (1-t)\delta + tg$.  The Hodge star of $A(g)$ with respect to $g$ is the smooth counterpart of what we defined in Definition~\ref{def:distconn2}.  We identity $\star A(g)$ with an element of $W'$ via
\[
\langle \star A(g), \alpha \rangle_{W',W} = \int_\Omega \langle \star A(g), \alpha \rangle_g \, \omega(g).
\]

\begin{theorem} \label{thm:conv}
Let $g$ be a smooth Riemannian metric on $\overline{\Omega}$.  
Let $\{g_h\}_{h>0}$ be a sequence of Regge metrics satisfying $\lim_{h \rightarrow 0} \|g_h-g\|_{L^\infty(\Omega)} = 0$ and $\sup_{h>0} \max_{T \in \mathcal{T}_h} \|g_h\|_{W^{1,\infty}(T)} < \infty$.  Then there exists a constant $C$ independent of $h$ such that 
\begin{equation} \label{conv}
\begin{split}
\| \distcurv(g_h) - \kappa(g)\omega(g) \|_{\color{blue}V',h} &+ \| \Gamma_{\rm dist}(g_h) - \star A(g) \|_{\color{blue}W',h} \\ &\le C \left[ \left( \sum_{T \in \mathcal{T}_h} h_T^{-2} \|g_h-g\|_{L^2(T)}^2 \right)^{1/2} + \left( \sum_{T \in \mathcal{T}_h} |g_h-g|_{H^1(T)}^2 \right)^{1/2} \right]
\end{split}
\end{equation}
for every $h$ sufficiently small.
\end{theorem}

This theorem, {\color{blue}proved below}, shows that if $g_h$ converges to $g$ rapidly enough, then the distributional curvature two-form associated with $g_h$ converges to $\kappa(g)\,\omega(g)$ in $V'$, and the canonical distributional connection one-form associated with $g_h$ converges to $\star A(g)$ in $W'$.  One way to accomplish this is to take $g_h$ equal to a suitable interpolant of $g$ onto $\mathcal{M}^r_h$, the space of Regge metrics on $\Omega$ that are piecewise polynomial with respect to $\mathcal{T}_h$ of degree at most $r$.   
{\color{blue}Below we state a corollary of Theorem~\ref{thm:conv} that applies when the chosen interpolant satisfies an elementwise error estimate of the form
\begin{equation} \label{interperror}
\|g_h-g\|_{L^2(T)} + h_T |g_h-g|_{H^1(T)} \le c h_T^{r+1} |g|_{H^{r+1}(T)}
\end{equation}
for some constant $c$ depending only on $r$ and the ratio $h_T/\rho_T$.
}

\begin{corollary} \label{cor:conv}
If $r \ge 1$, $h$ is sufficiently small, and $g_h \in \mathcal{M}_h^r$ {\color{blue}satisfies~(\ref{interperror})}, then
\[
\| \distcurv(g_h) - \kappa(g)\omega(g) \|_{\color{blue}V',h} + \| \Gamma_{\rm dist}(g_h) - \star A(g) \|_{\color{blue}W',h} \le C \left( \sum_{T \in \mathcal{T}_h} h_T^{2r} |g|_{H^{r+1}(T)}^2 \right)^{1/2}.
\]
\end{corollary}

{\color{blue}
\begin{remark}
Under a stronger assumption on $g_h$, Gopalakrishnan, Neunteufel, Sch\"{o}berl, and Wardetzky~\cite{gopalakrishnan2022analysis} have recently proved improved error estimates for (piecewise polynomial projections of) the distributional curvature and distributional connection.  Their estimates require that $g_h$ be the canonical interpolant of $g$ onto $\mathcal{M}_h^r$~\cite[p. 29]{li2018regge}, which is an interpolant that we will discuss in more detail in Section~\ref{sec:poly}.  By exploiting subtle properties of this interpolant, they derive error estimates that are one order higher in $h$ and include the case $r=0$.

To be clear, we presented error estimates for the distributional curvature and distributional connection in the $V'$-norm and $W'$-norm, whereas~\cite{gopalakrishnan2022analysis} presents error estimates in the $H^{-1}(\Omega)$-norm and $L^2(\Omega)$-norm for \emph{projections} of the aforementioned quantities onto piecewise polynomial finite element spaces (just like~\cite{gawlik2019high} did for the curvature).  One can relate the two versions of the estimates without much difficulty by estimating the projection error; see~\cite{gawlik2019high} and~\cite{gopalakrishnan2022analysis}.
\end{remark}  
\begin{remark}
Under a different assumption on $g_h$, Cheeger, M\"{u}ller, and Schrader~\cite{cheeger1984curvature} proved that when $r=0$, the distributional curvature two-form of $g_h$ converges in the (setwise) sense of measures to the curvature two-form of $g$ at a rate of $O(h)$ in two dimensions~\cite[Equation (5.7)]{cheeger1984curvature} and at a rate of $O(h^{1/2})$ in three and higher dimensions~\cite[Theorem 5.1]{cheeger1984curvature}.  Their estimate requires that $g_h$ be the unique piecewise flat Regge metric with the property that on each $T \in \mathcal{T}_h$, the lengths of the edges of $T$, as measured by $g_h$, agree with the geodesic distances between the vertices of $T$, as measured by $g$.
\end{remark}
}

Let us prove Theorem~\ref{thm:conv}.  
Consider the one-parameter family of Regge metrics $G_h(t) = (1-t)\delta + tg_h$ emanating from the Euclidean metric $\delta$ at $t=0$.  Since $\distcurv(\delta)=0$, $\frac{\partial}{\partial t} G_h = g_h-\delta$, and $G_h(1)=g_h$, Theorem~\ref{thm:main} implies that
\begin{equation} \label{kdistint}
\langle \distcurv(g_h), v \rangle_{V',V} = \frac{1}{2} \int_0^1 b_h((1-t)\delta + tg_h; g_h-\delta, v) \, dt, \quad \forall v \in V.
\end{equation}
On the other hand, Proposition~\ref{prop:kappavoldot} implies that the curvature two-form $\kappa(g)\omega(g)$ satisfies
\[
\langle \kappa(g)\omega(g), v \rangle_{V',V} = \frac{1}{2} \int_0^1 b((1-t)\delta + tg; g-\delta, v) \, dt, \quad \forall v \in V,
\]
where 
\[
b(g;\sigma,v) = \langle \dv \dv S \sigma, v \rangle_{g,\mathcal{S}}.
\]
In view of Proposition~\ref{prop:bhrewritten}, $b(g;\sigma,v)$ coincides with $b_h(g;\sigma,v)$ for smooth $g$ and $\sigma$, so we may replace $b$ by $b_h$:
\[
\langle \kappa(g)\omega(g), v \rangle_{V',V} = \frac{1}{2} \int_0^1 b_h((1-t)\delta + tg; g-\delta, v) \, dt, \quad \forall v \in V.
\]
It follows that for any $v \in V$,
\begin{align*}
\langle \distcurv(g_h) - \kappa(g)\omega(g), v \rangle_{V',V} 
&= \frac{1}{2} \int_0^1 \left[ b_h((1-t)\delta + tg_h; g_h-\delta, v) - b_h((1-t)\delta + tg; g-\delta, v) \right] \, dt \\
&= \frac{1}{2} \int_0^1 \left[ b_h((1-t)\delta + tg_h; g_h-\delta, v) - b_h((1-t)\delta + tg; g_h-\delta, v) \right] \, dt \\
&\;\;\; + \frac{1}{2} \int_0^1 b_h((1-t)\delta + tg; g_h-g, v) \, dt.
\end{align*}
Thus,
\begin{align}
\left| \langle \distcurv(g_h) - \kappa(g)\omega(g), v \rangle_{V',V} \right|
&\le \frac{1}{2} \int_0^1 \left| b_h((1-t)\delta + tg_h; g_h-\delta, v) - b_h((1-t)\delta + tg; g_h-\delta, v) \right| \, dt \nonumber \\
&\;\;\; + \frac{1}{2} \int_0^1 \left| b_h((1-t)\delta + tg; g_h-g, v) \right| \, dt. \label{boundedbybh}
\end{align}

Arguing similarly for $\Gamma_{\rm dist}(g_h)$, we find that for any $\alpha \in W$,
\begin{align}
\left| \langle \Gamma_{\rm dist}(g_h) - \star A(g), \alpha \rangle_{W',W} \right|
&\le \frac{1}{2} \int_0^1 \left| c_h((1-t)\delta + tg_h; g_h-\delta, \alpha) - c_h((1-t)\delta + tg; g_h-\delta, \alpha) \right| \, dt \nonumber \\
&\;\;\; + \frac{1}{2} \int_0^1 \left| c_h((1-t)\delta + tg; g_h-g, \alpha) \right| \, dt. \label{boundedbych}
\end{align}
\begin{lemma} \label{lemma:bhbound}
We have
\begin{equation} \label{bhbound}
\begin{split}
&\left| b_h((1-t)\delta + tg_h; g_h-\delta, v) - b_h((1-t)\delta + tg; g_h-\delta, v) \right| + \left| b_h((1-t)\delta + tg; g_h-g, v) \right| \\
&\le C\left[ \left( \sum_{T \in \mathcal{T}_h} h_T^{-2} \|g_h-g\|_{L^2(T)}^2 \right)^{1/2} + \left( \sum_{T \in \mathcal{T}_h} |g_h-g|_{H^1(T)}^2 \right)^{1/2} \right] \left[ |v|_{H^1(\Omega)} + \left( \sum_{T \in \mathcal{T}_h} h_T^2 |v|_{H^2(T)}^2 \right)^{1/2} \right]
\end{split}
\end{equation}
and
\begin{equation} \label{chbound}
\begin{split}
&\left| c_h((1-t)\delta + tg_h; g_h-\delta, \alpha) - c_h((1-t)\delta + tg; g_h-\delta, \alpha) \right| + \left| c_h((1-t)\delta + tg; g_h-g, \alpha) \right| \\
&\le C\left[ \left( \sum_{T \in \mathcal{T}_h} h_T^{-2} \|g_h-g\|_{L^2(T)}^2 \right)^{1/2} + \left( \sum_{T \in \mathcal{T}_h} |g_h-g|_{H^1(T)}^2 \right)^{1/2} \right] \left[ \|\alpha\|_{L^2(\Omega)} + \left( \sum_{T \in \mathcal{T}_h} h_T^2 |\alpha|_{H^1(T)}^2 \right)^{1/2} \right]
\end{split}
\end{equation}
for every $t \in [0,1]$ and every $h$ sufficiently small.
\end{lemma}
\begin{proof}
The first inequality is a direct application of Lemmas 4.11-4.12 of~\cite{gawlik2019high}.  Our situation is identical to the one there, except that we have not assumed that $v$ and $g_h$ are piecewise polynomial, that $\mathcal{T}_h$ is quasi-uniform, nor that $\lim_{h \rightarrow 0} h^{-1}\log h^{-1} \|g_h-g\|_{L^2(\Omega)} \rightarrow 0$.  It is easy to check that these discrepancies are immaterial in the context of~\cite[Lemmas 4.11-4.12]{gawlik2019high}.  Note that the upper bounds in~\cite[Lemmas 4.11-4.12]{gawlik2019high} are written in a slightly simpler form there using the bounds $h_T^{-1} \le Ch^{-1}$ (from quasi-uniformity, which was assumed in~\cite{gawlik2019high} but not here) and $h_T \le h$.

The inequality~(\ref{chbound}) is obtained by replacing all instances of $dv$ by $\alpha$ in the proof of Lemmas 4.11-4.12 in~\cite{gawlik2019high}.
\end{proof}

By combining Lemma~\ref{lemma:bhbound} with the bounds~(\ref{boundedbybh}-\ref{boundedbych}) and the definitions of $\|\cdot\|_{\color{blue}V,h}$, $\|\cdot\|_{\color{blue}V',h}$, $\|\cdot\|_{\color{blue}W,h}$, and $\|\cdot\|_{\color{blue}W',h}$, we arrive at Theorem~\ref{thm:conv}.  Corollary~\ref{cor:conv} then follows from interpolation error estimates for piecewise polynomial Regge metrics~\cite[Theorem 2.5]{li2018regge}.

\section{Piecewise polynomial setting} \label{sec:poly}

So far we have discussed curvature and connections for Regge metrics in the distributional sense.  For practical computing, it is often desirable to work with piecewise polynomial projections of these quantities.  We define such quantities below, and we show that the associated projection operators and linearized differential operators fit nicely into a commutative diagram of differential complexes.

We first define a few finite element spaces.  For each integer $r \ge 0$, let $\mathcal{P}_r(T)$ denote the space of polynomials of degree at most $r$ on a triangle $T$.  Let $\mathcal{P}_r\Lambda^k(T)$ denote the space of $k$-forms on $T$ with coefficients in $\mathcal{P}_r(T)$, and let $\mathcal{P}_r S_2^0(T)$ denote the space of symmetric $(0,2)$-tensor fields on $T$ with coefficients in $\mathcal{P}_r(T)$.  Also let $\mathcal{P}_r^-\Lambda^k(T) = \{\alpha \in \mathcal{P}_r\Lambda^k(T) \mid \operatorname{kos} \alpha \in \mathcal{P}_r\Lambda^{k-1}(T)\}$, where $\operatorname{kos}$ denotes the Koszul differential~\cite[p. 329]{arnold2010finite}. 

With $r \ge 0$ fixed, we define finite element spaces
\begin{align*}
V_h^{r+1} &= \{ v \in V \mid \left.v\right|_T \in \mathcal{P}_{r+1}(T) \, \forall T \in \mathcal{T}_h \}, \\
W_h^{r+1} &= \{\alpha \in W \mid \left.\alpha\right|_T \in \mathcal{P}_{r+1}^-\Lambda^1(T) \, \forall T \in \mathcal{T}_h \}, \\
X_h^r &= \{F \in X \mid \left.F\right|_T \in \mathcal{P}_r\Lambda^2(T) \, \forall T \in \mathcal{T}_h \}, \\
\Sigma_h^r &= \{\sigma \in \Sigma \mid \left.\sigma\right|_T \in \mathcal{P}_r S^0_2(T) \, \forall T \in \mathcal{T}_h \}, \\
\mathcal{M}_h^r &= \{g \in \mathcal{M} \mid \left.g\right|_T \in \mathcal{P}_r S^0_2(T) \, \forall T \in \mathcal{T}_h \}.
\end{align*}

On a triangulation of a planar domain, the space $V_h^{r+1}$ is the standard Lagrange finite element space consisting of continuous functions that are piecewise polynomial.  The space $W_h^{r+1}$ is isomorphic (via the identification of one-forms with vector fields) to the space of two-dimensional N\'ed\'elec finite elements of the first kind, whose members have single-valued tangential components along edges.  The space $X_h^r$ is isomorphic (via the identification of two-forms with scalar fields) to the space of discontinuous polynomials.  The space $\Sigma_h^r$ is the space of Regge finite elements: symmetric $(0,2)$-tensor fields that are piecewise polynomial and possess single-valued tangential-tangential components along edges.  The space $\mathcal{M}_h^r$ consists of piecewise polynomial Regge metrics; it is the subset of $\Sigma_h^r$ whose members are positive definite everywhere.

Note that $dV_h^{r+1} \subset W_h^{r+1}$ and $dW_h^{r+1} \subset X_h^r$.  In fact, the complex
\begin{center}
  \begin{tikzcd}
    0 \arrow{r}{} & V_h^{r+1} \arrow{r}{d} & W_h^{r+1} \arrow{r}{d} & X_h^r \arrow{r}{} & 0
    
  \end{tikzcd}
\end{center}
is exact on triangulations of contractible planar domains~\cite[Section 5.5]{arnold2010finite}.  We will not make extensive use of this complex in what follows, except in the {\color{blue}paragraph preceding Section~\ref{sec:complexes}}; our main interest is in a different complex involving $\Sigma_h^r$. But it is worth remarking that a commutative diagram
\[
  \begin{tikzcd}
    0 \arrow{r}{} & V \arrow{r}{d} \arrow[swap]{d}{} & W \arrow{d}{} \arrow{r}{d} & X \arrow{d}{} \arrow{r}{} & 0 \\
     0 \arrow{r}{} & V_h^{r+1} \arrow{r}{d} & W_h^{r+1} \arrow{r}{d} &  X_h^r \arrow{r}{} & 0
  \end{tikzcd}
\]
can be constructed using the canonical interpolation operators of finite element exterior calculus in the vertical arrows above~\cite{arnold2006finite,arnold2010finite}.

The following definitions should be regarded as piecewise polynomial versions of Definitions~\ref{def:distcurv} and~\ref{def:distconn}.

\begin{definition} \label{def:disccurv}
Let $g$ be a Regge metric.  The \emph{discrete Gaussian curvature} of $g$ is the function $\kappa_h(g) \in V_h^{r+1}$ defined by
\begin{equation} \label{disccurv}
\langle \kappa_h(g), v \rangle_{g,\mathcal{S}} = \sum_{T \in \mathcal{T}_h} \langle \kappa_T(g), v \rangle_{g,T} + \sum_{e \in \mathcal{E}_h^0} \langle \llbracket k_e(g) \rrbracket, v \rangle_{g,e} + \sum_{z \in \mathcal{V}_h^0} \Theta_z(g) v(z), \quad \forall v \in V_h^{r+1}.
\end{equation}
Equivalently,
\begin{equation} \label{curvproject}
\langle \kappa_h(g), v \rangle_{g,\mathcal{S}} = \langle \distcurv(g), v \rangle_{V',V}, \quad \forall v \in V_h^{r+1}.
\end{equation}
\end{definition}

{\color{blue}
\begin{definition} \label{def:discconn}
Let $g$ be a Regge metric.  
Let $(e_1,e_2)$ be a $g$-orthonormal frame field that is smooth on each triangle $T \in \mathcal{T}_h$.
The \emph{discrete connection one-form} associated with $(e_1,e_2)$ is the one-form $\Gamma_h(g) \in W_h^{r+1}$ defined by
\[
\langle \Gamma_h(g), \alpha \rangle_{g,\mathcal{S}} = \sum_{T \in \mathcal{T}_h} \langle \star A_T(g), \alpha \rangle_{g,T} - \sum_{e \in \mathcal{E}_h^0} \langle \theta_e(g), \alpha(\tau) \rangle_{g,e}, \quad \forall \alpha \in W_h^{r+1}.
\]
Equivalently,
\begin{equation} \label{connproject}
\langle \Gamma_h(g), \alpha \rangle_{g,\mathcal{S}} = \langle \Gamma_{\rm dist}(g), \alpha \rangle_{W',W}, \quad \forall \alpha \in W_h^{r+1}.
\end{equation}
\end{definition}
}

{\color{dgreen}
  As in Definition~\ref{def:distconn}, the distributional connection requires not just a choice of frame field $(e_1,e_2)$ but also a topologically consistent way of choosing $\theta_e$, which a priori is defined only up to $2\pi$; recall the discussion preceding Definition~\ref{def:distconn}.
}

{\color{blue}
\begin{remark}
Note the distinction between $\distcurv(g)$ and $\kappa_h(g)$: the former is an element of $V'$ whereas the latter is a continuous, piecewise polynomial function.  Similarly, $\Gamma_{\rm dist}(g)$ is an element of $W'$ whereas $\Gamma_h(g)$ is a piecewise polynomial $1$-form.  One can see from~(\ref{curvproject}) that computing $\kappa_h(g)$ from $\distcurv(g)$ involves inverting a mass matrix.  Similarly, one can see from~(\ref{connproject}) that the same is true for computing $\Gamma_h(g)$ from $\Gamma_{\rm dist}(g)$.
\end{remark}
}

\begin{remark}
In view of~(\ref{kdistint}), an equivalent definition of $\kappa_h(g)$ on a planar triangulation is 
\begin{equation} \label{disccurvint}
\langle \kappa_h(g), v \rangle_{g,\mathcal{S}} = \frac{1}{2} \int_0^1 b_h((1-t)\delta + tg; g-\delta, v) \, dt, \quad \forall v \in V_h^{r+1}.
\end{equation}
This is precisely the definition of discrete Gaussian curvature that was proposed in~\cite{gawlik2019high}.  The discovery that~(\ref{disccurvint}) and~(\ref{disccurv}) are equivalent is one of the main contributions of the present paper.
\end{remark}

{\color{blue}
The following are immediate consequences of~(\ref{curvproject}),~(\ref{connproject}), and Propositions~\ref{prop:ddistconn},~\ref{prop:framechange}, and~\ref{prop:distconndot}.
\begin{proposition} \label{prop:ddiscconn}
The discrete exterior coderivative of $\Gamma_h(g)$ is $-\kappa_h(g)$.  That is,
\[
\langle \Gamma_h(g), dv \rangle_{g,\mathcal{S}} = -\langle \kappa_h(g), v \rangle_{g,\mathcal{S}}, \quad \forall v \in V_h^{r+1}.
\]
\end{proposition}
\begin{proposition} \label{prop:discframechange}
{\color{dgreen}Let $g$ be a Regge metric, and let $(\bar{e}_1,\bar{e}_2)$ be a piecewise smooth $g$-orthonormal frame field with discrete connection one-form $\bar{\Gamma}_h(g) \in W_h^{r+1}$. Let $\psi$ be a piecewise smooth scalar field (not necessarily continuous), let $(e_1,e_2)$ be the piecewise smooth $g$-orthonormal frame field obtained by rotating $(\bar{e}_1,\bar{e}_2)$ counterclockwise by $\psi$, and let $\Gamma_h(g) \in W_h^{r+1}$ be the corresponding discrete connection one-form.} We have
\[
\langle \Gamma_h(g) - \bar{\Gamma}_h(g), \alpha \rangle_{g,\mathcal{S}} = -\langle \psi_h, d\alpha \rangle_{g,\mathcal{S}}, \quad \forall \alpha \in W_h^{r+1},
\]
where $\psi_h \in X_h^r$ satisfies $\langle \psi_h, F \rangle_{g,\mathcal{S}} = \sum_{T \in \mathcal{T}_h} \int_T \psi F$ for all $F \in X_h^r$.
\end{proposition}
\begin{proposition} \label{prop:discconndot}
Let $g(t)$ be a Regge metric depending smoothly on $t$.   Let $(e_1(t),e_2(t))$ be a $g(t)$-orthonormal frame field that is smooth on each triangle $T \in \mathcal{T}_h$ and depends smoothly on $t$.  Then
\[
\frac{d}{dt} \langle \Gamma_h(g(t)), \alpha \rangle_{g,\mathcal{S}} = -\frac{1}{2} c_h(g(t); \sigma(t), \alpha) - \frac{d}{dt} \langle F_h(t), d\alpha \rangle_{g,\mathcal{S}}, \quad \forall \alpha \in W_h^{r+1},
\]
where $\sigma = \frac{\partial}{\partial t}g$ and $F_h(t) \in X_h^r$ is given by
\[
\langle F_h(t), G \rangle_{g,\mathcal{S}} = \sum_{T \in \mathcal{T}_h} \int_T \widetilde{f}(t) G, \quad \forall G \in X_h^r.
\]
Here, $\widetilde{f}(t) = \int_0^t f(t')\,dt'$, and $f = \frac{1}{2}\sigma(e_1,e_2) - \dot{e}^2(e_1) = -\frac{1}{2}\sigma(e_1,e_2) + \dot{e}^1(e_2)$.
\end{proposition}
}

{\color{blue}
\paragraph{Choosing a frame field.}
Just like in Section~\ref{sec:distconn}, Definition~\ref{def:discconn} assumes that one has selected a $g$-orthonormal frame field $(e_1,e_2)$ in advance.  On a triangulated planar domain, we can construct such a frame field as we did in Section~\ref{sec:distconn}, leading to a \emph{canonical discrete connection one-form} $\Gamma_h(g) \in W_h^{r+1}$ defined by
\begin{equation} \label{canonicaldiscconn}
\langle \Gamma_h(g), \alpha \rangle_{g,\mathcal{S}} = -\frac{1}{2} \int_0^1 c_h\left( (1-t)\delta + tg; g-\delta, \alpha \right) \, dt, \quad \forall \alpha \in W_h^{r+1}.
\end{equation}
This is a piecewise polynomial version of Definition~\ref{def:distconn2}.  Recall the interpretation of this one-form. Among all possible discrete connection one-forms that can be constructed with Definition~\ref{def:discconn}, the canonical discrete connection one-form~(\ref{canonicaldiscconn}) is the one that is associated with a specific frame field constructed as follows: starting from a constant frame field that is orthonormal with respect to the Euclidean metric, we deform the metric and the frame field until the frame field is orthonormal with respect to $g$.  Tracking the evolution of the connection along the way yields~(\ref{canonicaldiscconn}).

An alternative way to (implicitly) single out a frame field is to find a one-form $\Gamma_h \in W_h^{r+1}$ that solves the Hodge--Dirac problem~\cite{leopardi2016abstract}
\begin{align}
\langle \Gamma_h, dv \rangle_{g,\mathcal{S}} &= -\langle \kappa_h(g), v \rangle_{g,\mathcal{S}}, && \forall v \in V_h^{r+1}, \label{hodgedirac1} \\
\langle d\Gamma_h, G \rangle &= 0, && \forall G \in X_h^r. \label{hodgedirac2}
\end{align}
{\color{dgreen} This approach is motivated by the Coulomb gauge condition that chooses a gauge by requiring that the divergence of the connection vector field be zero.} We now show that there exists at least one solution of~(\ref{hodgedirac1}--\ref{hodgedirac2}) that is a valid discrete connection one-form, and this solution is unique if the domain is contractible.  Note that the canonical discrete connection one-form, hereafter denoted $\bar{\Gamma}_h$, satisfies~(\ref{hodgedirac1}) but not necessarily~(\ref{hodgedirac2}).  Setting $F_h := d\bar{\Gamma}_h \in X_h^r$, we can solve the Hodge--Laplace problem
\begin{align*}
\langle \beta_h, \alpha \rangle_{g,\mathcal{S}}  &= \langle \psi_h, d\alpha \rangle_{g,\mathcal{S}}, && \forall \alpha \in W_h^{r+1}, \\
\langle d\beta_h, G \rangle_{g,\mathcal{S}} &= \langle F_h, G \rangle_{g,\mathcal{S}}, && \forall G \in X_h^r, 
\end{align*}
for $(\beta_h,\psi_h) \in W_h^{r+1} \times X_h^r$, and then define $\Gamma_h \in W_h^{r+1}$ by
\begin{equation} \label{discconn_constructed}
\langle \Gamma_h, \alpha \rangle_{g,\mathcal{S}} = \langle \bar{\Gamma}_h, \alpha \rangle_{g,\mathcal{S}} - \langle \psi_h, d\alpha \rangle_{g,\mathcal{S}}, \quad \forall \alpha \in W_h^{r+1}.
\end{equation}
By construction, $\Gamma_h$ satisfies~(\ref{hodgedirac1}--\ref{hodgedirac2}).  {\color{dgreen} To see that $\Gamma_h$ yields a valid discrete connection one-form, we must show that it is associated with some $g$-orthonormal frame field $(e_1,e_2)$.  In view of Proposition~\ref{prop:discframechange}, we see that $\Gamma_h$ is the discrete connection one-form associated with a frame field $(e_1,e_2)$ that is rotated by $\star \psi_h$ relative to $(\bar{e}_1,\bar{e}_2)$, the frame field associated with $\bar{\Gamma}_h$.}
Note that this one-form $\Gamma_h$ is uniquely determined by~(\ref{hodgedirac1}-\ref{hodgedirac2}) when the domain is contractible, because then there are no discrete harmonic one-forms.  On non-contractible domains, equations~(\ref{hodgedirac1}-\ref{hodgedirac2}) only determine $\Gamma_h$ up to the addition of a discrete harmonic one-form, and not all solutions of~(\ref{hodgedirac1}-\ref{hodgedirac2}) are valid discrete connection one-forms.  A sufficient condition ensuring $\Gamma_h$'s validity is that~(\ref{discconn_constructed}) holds for some $\psi_h \in X_h^r$, which is equivalent to the condition that the harmonic part of $\Gamma_h$ coincides with the harmonic part of $\bar{\Gamma}_h$.
}

\subsection{Differential complexes} \label{sec:complexes}

The linearization of $\distcurv$ around a Regge metric $g$ is a differential operator {\color{blue}that maps} perturbations of $g$ to elements of $V'$.  This operator, together with the linearization of $\kappa_h$, fits into a commutative diagram of differential complexes which we describe below.  This diagram bears strong similarities to ones studied in~\cite{chen2018multigrid,christiansen2011linearization}; see also~\cite{christiansen2019finite,chen2021finite} for related complexes with higher regularity.

Throughout the following discussion, we let $g \in \mathcal{M}$ and $g_h \in \mathcal{M}_h^r$ be fixed Regge metrics.

\paragraph{Differential operators.}
We define operators $\dv_{\rm dist} : W' \rightarrow V'$ and $(\dv S)_{\rm dist} : \Sigma \rightarrow W'$ by
\begin{align*}
\langle \dv_{\rm dist} \alpha, v \rangle_{V',V} &= -\langle \alpha, dv \rangle_{W',W}, \quad \forall \alpha \in W', \, v \in V, \\
\langle (\dv S)_{\rm dist} \sigma, \alpha \rangle_{W',W} &= -c_h(g; \sigma, \alpha), \quad \forall \sigma \in \Sigma, \, \alpha \in W.
\end{align*}
We also define $(\dv \dv S)_{\rm dist} = \dv_{\rm dist} (\dv S)_{\rm dist}$, which is a map from $\Sigma$ to $V'$.  By construction, we have
\begin{equation}\label{eq:divdivSdist}
  \begin{split}
    \langle (\dv \dv S)_{\rm dist} \sigma, v \rangle_{V',V} 
    &= -\langle (\dv S)_{\rm dist} \sigma, dv \rangle_{\color{blue}W',W} \\
    &= c_h(g; \sigma, dv) \\
    &= b_h(g; \sigma, v), \quad \forall \sigma \in \Sigma, \, v \in V,
  \end{split}
\end{equation}
which shows that $(\dv \dv S)_{\rm dist}$ is the linearization of $2\distcurv$ around $g$.

We define analogous operators on the finite element spaces.  Namely, $\dv_h : W_h^{r+1} \rightarrow V_h^{r+1}$ and $(\dv S)_h : \Sigma_h^r \rightarrow W_h^{r+1}$ are defined by 
\begin{align*}
\langle \dv_h \alpha, v \rangle_{g_h,\mathcal{S}} &= -\langle \alpha, dv \rangle_{g_h,\mathcal{S}}, \quad \forall \alpha \in W_h^{r+1}, \, v \in V_h^{r+1}, \\
\langle (\dv S)_h \sigma, \alpha \rangle_{g_h,\mathcal{S}} &= -c_h(g_h; \sigma, \alpha), \quad \forall \sigma \in \Sigma_h^r, \, \alpha \in W_h^{r+1}.
\end{align*}
Denoting $(\dv \dv S)_h = \dv_h (\dv S)_h : \Sigma_h^r \rightarrow V_h^{r+1}$, we have
\begin{align*}
\langle (\dv \dv S)_h \sigma, v \rangle_{g_h,\mathcal{S}} 
&= -\langle (\dv S)_h \sigma, dv \rangle_{g_h,\mathcal{S}}  \\
&= c_h(g_h; \sigma, dv) \\
&= b_h(g_h; \sigma, v), \quad \forall \sigma \in \Sigma_h^r, \, v \in V_h^{r+1},
\end{align*}
so $(\dv \dv S)_h$ is the linearization of $2\kappa_h \omega_h$ around $g_h$, where $\omega_h$ denotes the volume form associated with $g_h$.

\paragraph{Projections.}
We define projectors $\pi_h^V : V' \rightarrow V_h^{r+1}$ and $\pi_h^W : W' \rightarrow W_h^{r+1}$ by
\begin{align*}
\langle \pi_h^V u, v \rangle_{g_h,\mathcal{S}} &= \langle u, v \rangle_{V',V}, \quad \forall u \in V', \, v \in V_h^{r+1}, \\
\langle \pi_h^W\alpha, \beta \rangle_{g_h,\mathcal{S}} &= \langle \alpha, \beta \rangle_{W',W}, \quad \forall \alpha \in W', \, \beta \in W_h^{r+1}.
\end{align*}

To define an interpolation operator onto $\Sigma_h^r$, it will be convenient to fix a piecewise constant Regge metric $\overline{g}_h \in \mathcal{M}_h^0$.  We define $\pi_h^{\Sigma} : \Sigma \rightarrow \Sigma_h^r$ by requiring that for $\sigma \in \Sigma$, the interpolant $\pi_h^{\Sigma}\sigma \in \Sigma_h^r$ satisfies
\begin{align}
\langle \pi_h^{\Sigma} \sigma - \sigma, \rho \rangle_{\overline{g}_h,T} &= 0, \quad \forall \rho \in \mathcal{P}_{r-1}S_2^0(T), \, T \in \mathcal{T}_h, \label{Ihdof1} \\
\langle (\pi_h^{\Sigma} \sigma)(\overline{\tau}_h, \overline{\tau}_h) - \sigma(\overline{\tau}_h, \overline{\tau}_h), v \rangle_{\overline{g}_h,e} &= 0, \quad \forall v \in \mathcal{P}_r(e), \, e \in \mathcal{E}_h, \label{Ihdof2}
\end{align}
where $\overline{\tau}_h$ is the unit tangent with respect $\overline{g}_h$.  This interpolation operator onto $\Sigma_h^r$ was introduced in~\cite[p. 29]{li2018regge}.  If $\mathcal{S}$ triangulates a planar domain and $\overline{g}_h=\delta$, then $\pi_h^{\Sigma}$ is the canonical interpolation operator onto $\Sigma_h^r$ referenced in Corollary~\ref{cor:conv}.  It maps $\sigma$ into $\mathcal{M}_h^r$ if $h$ is sufficiently small and $\sigma$ belongs to $\mathcal{M}$.

Note that~(\ref{Ihdof1}) is equivalent to the condition that
\begin{equation} \label{Ihdof1b} 
\langle \overline{S}_h (\pi_h^{\Sigma} \sigma - \sigma), \rho \rangle_{\overline{g}_h,T} = 0, \quad \forall \rho \in \mathcal{P}_{r-1}S_2^0(T), \, T \in \mathcal{T}_h, 
\end{equation}
where $\overline{S}_h \sigma = \sigma - \overline{g}_h \Tr \sigma$ and the trace is taken with respect to $\overline{g}_h$.  This follows from two observations.  First, $\overline{S}_h$ is an involution that maps $\mathcal{P}_{r-1}S_2^0(T)$ to itself, so it is an automorphism of $\mathcal{P}_{r-1}S_2^0(T)$.  Second, we have $ \langle \sigma, \overline{S}_h \rho \rangle_{\overline{g}_h,T} = \langle \overline{S}_h \sigma, \rho \rangle_{\overline{g}_h,T}$ for all $\rho,\sigma \in \Sigma$.

In view of~(\ref{Ihdof2}),~(\ref{Ihdof1b}) and the definition~(\ref{bh}) of $b_h$, we have
\begin{equation}
b_h(\overline{g}_h; \pi_h^{\Sigma} \sigma-\sigma, v) = 0, \quad \forall \sigma \in \Sigma, \, v \in V_h^{r+1}.
\end{equation}

\paragraph{Additional definitions on planar domains.}
In the event that $\mathcal{S}$ triangulates a planar domain $\Omega$ and $g$ is smooth, we also introduce additional spaces and operators.  We define
\begin{align*}
U &= \{u \in H^1(\Omega) \otimes \mathbb{R}^2 \mid \left.u\right|_T \in  H^2(T) \otimes \mathbb{R}^2, \, \forall T \in \mathcal{T}_h \}, \\
U_h^{r+1} &= \{u \in U \mid \left.u\right|_T \in \mathcal{P}_{r+1}(T) \otimes \mathbb{R}^2, \, \forall T \in \mathcal{T}_h \}.
\end{align*}
We let $\df : U \rightarrow \Sigma$ be the differential operator
\[
\df u = \frac{1}{2} \mathcal{L}_u g.
\]
In coordinates~\cite[p. 12]{chow2006hamilton},
\[
(\df u)_{ij} = \frac{1}{2} \left(  (\nabla_i u)_j + (\nabla_j u)_i \right).
\]
Note that if $u \in U$, then our assumption that $g$ is smooth ensures that $\df u$ is well-defined, and it belongs to $\Sigma$ for the following reason.  On any edge $e$ shared by two triangles, the trace of $\nabla_{\tau} u$ is well-defined and single-valued, so $(\df u)(\tau,\tau) = g(\nabla_{\tau} u,\tau)$ is as well.

For a fixed $\overline{g}_h \in \mathcal{M}_h^0$, we define an interpolation operator $\pi_h^U : U \rightarrow U_h^{r+1}$ by requiring that for $u \in U$, the interpolant $\pi_h^U u \in U_h^{r+1}$ satisfies
\begin{align*}
\langle \pi_h^U u - u, v \rangle_{\overline{g}_h,T} &= 0, \quad \forall v \in \mathcal{P}_{r-2}(T) \otimes \mathbb{R}^2, \, T \in \mathcal{T}_h, \\
\langle \pi_h^U u - u, v \rangle_{\overline{g}_h, e} &= 0, \quad \forall v \in \mathcal{P}_{r-1}(e) \otimes \mathbb{R}^2, \, e \in \mathcal{E}_h, \\
\pi_h^U u(z) - u(z) &= 0, \quad \forall z \in \mathcal{V}_h.
\end{align*}
This is the standard Lagrange interpolation operator for continuous, piecewise polynomial vector fields of degree $r+1$.

\begin{theorem}
Let $g \in \mathcal{M}$, $g_h \in \mathcal{M}_h^r$, and $\overline{g}_h \in \mathcal{M}_h^0$.  The following statements hold:
\begin{enumerate}
\item The diagram
\[
  \begin{tikzcd}
    W' \arrow{r}{\dv_{\rm dist}} \arrow[swap]{d}{\pi_h^W} & V' \arrow{d}{\pi_h^V} \\
     W_h^{r+1} \arrow{r}{\dv_h} & V_h^{r+1}
  \end{tikzcd}
\]
commutes.
\item If $g=g_h=\overline{g}_h$, then the diagram
\[
  \begin{tikzcd}
    \Sigma \arrow{r}{(\dv\dv S)_{\rm dist}} \arrow[swap]{d}{\pi_h^{\Sigma}} &[2em] V' \arrow{d}{\pi_h^V} \\
     \Sigma_h^r \arrow{r}{(\dv\dv S)_h} &[2em] V_h^{r+1}
  \end{tikzcd}
\]
commutes.
\item If $\mathcal{S}$ is planar and $g=g_h=\overline{g}_h=\delta$, then the diagram
\[
  \begin{tikzcd}
     U \arrow{r}{\df} \arrow[swap]{d}{\pi_h^U} & \Sigma \arrow{d}{\pi_h^{\Sigma}} \\
     U_h^{r+1} \arrow{r}{\df} & \Sigma_h^r
  \end{tikzcd}
\]
commutes.
\end{enumerate}
\end{theorem}
\begin{proof}
\begin{enumerate}
\item For any $\alpha \in W'$ and any $v \in V_h^{r+1}$, we have
\begin{align*}
\langle \dv_h \pi_h^W \alpha, v \rangle_{g_h,\mathcal{S}}
&= -\langle \pi_h^W \alpha, dv \rangle_{g_h,\mathcal{S}} \\
&= -\langle \alpha, dv \rangle_{W',W} \\
&= \langle \dv_{\rm dist} \alpha, v \rangle_{V',V} \\
&= \langle \pi_h^V \dv_{\rm dist} \alpha, v \rangle_{g_h,\mathcal{S}},
\end{align*}
so $\dv _h \pi_h^W = \pi_h^V \dv_{\rm dist}$.
\item If $g=g_h = \overline{g}_h$, then for any $\sigma \in \Sigma$ and any $v \in V_h^{r+1}$, we have
\begin{align*}
\langle (\dv\dv S)_h \pi_h^{\Sigma} \sigma, v \rangle_{g_h,\mathcal{S}} 
&= b_h(g_h; \pi_h^{\Sigma} \sigma, v) \\
&= b_h(\overline{g}_h; \pi_h^{\Sigma} \sigma, v) \\
&= b_h(\overline{g}_h; \sigma, v) \\
&= b_h(g; \sigma, v) \\
&= \langle (\dv\dv S)_{\rm dist} \sigma, v \rangle_{V',V} \\
&= \langle \pi_h^V (\dv\dv S)_{\rm dist} \sigma, v \rangle_{g_h,\mathcal{S}},
\end{align*}
so $(\dv\dv S)_h \pi_h^{\Sigma}  = \pi_h^V (\dv\dv S)_{\rm dist}$.
\item If $\mathcal{S}$ is planar and $g=g_h = \overline{g}_h = \delta$, then consider an arbitrary $u \in U$.  On each triangle $T \in \mathcal{T}_h$, the definitions of $\pi_h^{\Sigma}$ and $\pi_h^U$ imply that for any $\rho \in \mathcal{P}_{r-1}S^0_2(T)$,
\begin{align*}
\langle \pi_h^{\Sigma} \df u, \rho \rangle_{\delta,T}
&= \langle \df u, \rho \rangle_{\delta,T} \\
&= \langle u, \rho n \rangle_{\delta,\partial T} - \langle u, \dv \rho \rangle_{\delta, T} \\
&= \langle \pi_h^U u, \rho n \rangle_{\delta,\partial T} - \langle \pi_h^U u, \dv \rho \rangle_{\delta, T} \\
&= \langle \df \pi_h^U u, \rho \rangle_{\delta,T}.
\end{align*}
On each edge $e \in \mathcal{E}_h$, the definitions of $\pi_h^{\Sigma}$ and $\pi_h^U$ imply that for any $v \in \mathcal{P}_r(e)$,
\begin{align*}
\langle (\pi_h^{\Sigma} \df u)(\tau,\tau), v \rangle_{\delta,e}
&= \langle (\df u)(\tau,\tau), v \rangle_{\delta,e} \\
&= \langle \nabla_{\tau} u, \tau v \rangle_{\delta,e} \\
&= \delta(u(z^{(2)}),\tau)v - \delta(u(z^{(1)}),\tau)v - \langle u, \tau  \nabla_{\tau} v \rangle_{\delta,e} \\
&= \delta(\pi_h^U u(z^{(2)}),\tau)v - \delta(\pi_h^U u(z^{(1)}),\tau)v - \langle \pi_h^U u, \tau  \nabla_{\tau} v \rangle_{\delta,e} \\
&= \langle (\df \pi_h^U u)(\tau,\tau), v \rangle_{\delta,e},
\end{align*}
where $z^{(1)},z^{(2)}$ are the two endpoints of $e$.  It follows that $\pi_h^{\Sigma} \df u = \df \pi_h^U u$.
\end{enumerate}
\end{proof}

When $\mathcal{S}$ is planar and $g=g_h=\overline{g}_h=\delta$, the theorem above can be summarized by saying that the diagram
\begin{center}
  \begin{tikzcd}
    RM \arrow{r}{\subset} & U \arrow{r}{\df} \arrow{d}{\pi_h^U} & \Sigma \arrow{d}{\pi_h^{\Sigma}} \arrow{r}{(\dv S)_{\rm dist}}  & W' \arrow[dotted]{d}{\pi_h^W} \arrow{r}{(\dv)_{\rm dist}}  & V' \arrow{d}{\pi_h^V} \arrow{r}{} & 0 \\
     RM \arrow{r}{\subset} & U_h^{r+1} \arrow{r}{\df} & \Sigma_h^r \arrow{r}{(\dv S)_h} & W_h^{r+1} \arrow{r}{(\dv)_h}  & V_h^{r+1} \arrow{r}{} & 0
  \end{tikzcd}
\end{center}
commutes if the dashed arrow is excluded or if the vertical arrows to its left are excluded.  Here, we introduced the space $RM =\operatorname{ker}(\df)$, which consists of vector fields on $\mathcal{S}$ of the form $u(x^1,x^2) = (a+bx^2,c-bx^1)$, where $a,b,c \in \mathbb{R}$.  The top and bottom rows are both complexes if we exclude the column containing $W'$ and $W_h^{r+1}$:
\begin{equation} \label{diagram}
  \begin{tikzcd}
    RM \arrow{r}{\subset} & U \arrow{r}{\df} \arrow{d}{\pi_h^U} & \Sigma \arrow{d}{\pi_h^{\Sigma}} \arrow{r}{(\dv \dv S)_{\rm dist}}  &[5em] V' \arrow{d}{\pi_h^V} \arrow{r}{} & 0 \\
     RM \arrow{r}{\subset} & U_h^{r+1} \arrow{r}{\df} & \Sigma_h^r \arrow{r}{(\dv \dv S)_h} & V_h^{r+1} \arrow{r}{} & 0
  \end{tikzcd}
\end{equation}
Indeed, we will show below that 
\begin{equation} \label{divdivSdef}
(\dv \dv S)_{\rm dist} \df = 0, \quad \text{ if } g=\delta.
\end{equation}
Since $(\dv \dv S)_h \df \pi_h^U = \pi_h^V (\dv \dv S)_{\rm dist} \df$, the surjectivity of $\pi_h^U$ implies that 
\begin{equation} \label{divdivSdefh}
\left. (\dv \dv S)_h \df \right|_{U_h^{r+1}} = 0, \quad \text{ if } g=g_h=\overline{g}_h=\delta
\end{equation}
as well.

{\color{blue}
\begin{remark}
The bottom row of~(\ref{diagram}) has a direct correspondence to the complex studied in~\cite[Equation 2.10]{chen2018multigrid}.  The spaces that are labelled $\overline{\boldsymbol{P}}_1(\Omega;\mathbb{R}^2)$, $\mathcal{S}_h$, $\mathcal{V}_h$, and $\mathcal{P}_h$ in~\cite[Equation 2.10]{chen2018multigrid} correspond in our notation to $J(RM)$, $J(U_h^{r+1})=U_h^{r+1}$, $S(\Sigma_h^r)$, and $V_h^{r+1}$, respectively, where $J$ denotes a $90^\circ$ rotation.  Furthermore, their differential operators $\nabla^s \times$ and $(\dv\boldsymbol{\dv})_h$ correspond to our $S\df J^{-1}$ and $(\dv\dv S)_h S^{-1}$, respectively. Thus,~\cite[Equation 2.10]{chen2018multigrid} reads
\begin{center}
  \begin{tikzcd}
    J(RM) \arrow{r}{\subset} & J(U_h^{r+1}) \arrow{r}{S\df J^{-1}} &[1em] S(\Sigma_h^r) \arrow{r}{(\dv\dv S)_h S^{-1}} &[3em] V_h^{r+1} \arrow{r}{} & 0
  \end{tikzcd}
\end{center}
in our notation.  

Because of this correspondence, it follows from~\cite[Lemma 2.6]{chen2018multigrid} that the bottom row of~(\ref{diagram}) is exact on contractible domains.  Exactness of the top row of~(\ref{diagram}) {\color{blue}can be studied using} a similar correspondence, although there is one subtlety: one must use an argument analogous to the one in Appendix~\ref{sec:appendix} in order to account for the fact that our infinite-dimensional spaces have higher elementwise regularity than global regularity.  This argument yields exactness at the positions $U$ and $\Sigma$ in the top row of~(\ref{diagram}).   Exactness at the position $V'$, i.e. surjectivity of $(\dv\dv S)_{\rm dist} : \Sigma \to V'$, does not appear to hold.  For example, on a triangulation consisting of a single triangle $T$, $(\dv\dv S)_{\rm dist}$ is not surjective from $H^1 S_2^0(T)$ to the dual of $H^2(T) \cap H^1_0(T)$.
\end{remark}
}

We emphasize that~(\ref{divdivSdef}) and (\ref{divdivSdefh}) only hold if $g=\delta$ and $g=g_h=\overline{g}_h=\delta$, respectively.  In the non-Euclidean setting, two obstructions emerge.  First, it is not clear how the spaces $U$ and $U_h^{r+1}$ should be defined for non-smooth metrics.  Second, even for smooth $g$, $(\dv \dv S)_{\rm dist} \df \neq 0$ in the presence of curvature.  Instead we have the following identities.

\begin{proposition} \label{lemma:dividistKu}
Let $g$ be a smooth Riemannian metric with Gaussian curvature $\kappa$.  For any $u \in U$, we have
\begin{equation} \label{divdistKu}
(\dv \dv S)_{\rm dist} \df u = \dv_{\rm dist} (\kappa u),
\end{equation}
where we view $\kappa u$ as an element of $W'$ via
\[
\langle \kappa u, \alpha \rangle_{W',W} = \int_\mathcal{S} \kappa \alpha(u) \omega, \quad \forall \alpha \in W.
\]
In particular,~(\ref{divdivSdef}) holds if $\kappa=0$.
\end{proposition}
\begin{proof}
We will first show that for any smooth vector field $u$,
\begin{equation} \label{divSdefu}
\dv S \df u = \kappa u^\flat - \frac{1}{2} d^* d u^\flat,
\end{equation}
where $d^*$ denotes the codifferential.  To see this, we compute
\[
\dv S \df u 
= \dv \df u - \dv(g\Tr\df u).
\]
A calculation in geodesic normal coordinates shows that
\[
\dv \df u = \dv \nabla u^\flat + \frac{1}{2}d^* du^\flat.
\]
On the other hand,
\[
\dv(g\Tr\df u) = d \Tr \df u = d \dv u = -dd^* u^\flat.
\]
Hence,
\begin{align*}
\dv S \df u 
&= \dv \nabla u^\flat + \frac{1}{2}d^* du^\flat + dd^* u^\flat \\
&= \dv \nabla u^\flat + (dd^*+d^* d) u^\flat - \frac{1}{2}d^* d u^\flat.
\end{align*}
The first two terms above are the difference between the Bochner and Hodge Laplacians of $u^\flat$.  The Weitzenbock formula gives
\[
\dv \nabla u^\flat + (dd^*+d^* d) u^\flat = \kappa u^\flat,
\]
so~(\ref{divSdefu}) follows.

Now we will consider a vector field $u \in U$ and prove~(\ref{divdistKu}).  
Since $u \big|_T \in H^2(T) \otimes \mathbb{R}^2$ on each $T \in \mathcal{T}_h$, and since the equality~(\ref{divSdefu}) extends to vector fields in $H^2(T) \otimes \mathbb{R}^2$ by density, we have that~(\ref{divSdefu}) holds elementwise.
Observe also that, by \eqref{eq:divdivSdist}, for any $v \in V$, we have
\[
\langle (\dv \dv S)_{\rm dist} \df u, v \rangle_{V',V} = b_h(g; \df u, v).
\]
Using~(\ref{bhgradv}), we see that 
\begin{equation} \label{bhdefu}
b_h(g; \df u, v) = \sum_{T \in \mathcal{T}_h} \left( -\int_T (\dv S \df u)(\nabla v) \, \omega + \int_{\partial T} (\df u)(n,\tau) \nabla_\tau v \, \ds \right).
\end{equation}
Using~(\ref{divSdefu}), the integrals over $T$ can each be rewritten as
\begin{align*}
-\int_T (\dv S \df u)(\nabla v) \, \omega 
&= -\langle \dv S \df u, dv \rangle_{g,T} \\
&= -\langle \kappa u^\flat, dv \rangle_{g,T} + \frac{1}{2} \langle d^* d u^\flat, dv \rangle_{g,T} \\
&= -\int_T \kappa dv(u) \omega -  \frac{1}{2}\int_{\partial T} \star du^\flat \wedge dv,
\end{align*}
where we used Stokes' theorem in the last line.  Inserting this into~(\ref{bhdefu}) and rewriting the second term, we get
\begin{equation} \label{bhdefu2}
b_h(g; \df u, v) = -\langle \kappa u, dv \rangle_{W',W} + \sum_{T \in \mathcal{T}_h} \int_{\partial T} \left( -\frac{1}{2} \star du^\flat \wedge dv + (\df u)(n,\tau) dv(\tau) \, \ds \right).
\end{equation}
One checks that the tangential component of the one-form being integrated above is
\begin{align*}
i_\tau \left( -\frac{1}{2} \star du^\flat \wedge dv + (\df u)(n,\tau) dv(\tau) \, \ds \right) 
&= \left( -\frac{1}{2} \star du^\flat + (\df u)(n,\tau) \right) dv(\tau)  \\
&= g(n,\nabla_\tau u) dv(\tau).
\end{align*}
Since $\nabla_\tau u$ and $dv(\tau)$ are single-valued on edges $e \in \mathcal{E}_h^0$, and since $dv(\tau)$ vanishes on $\partial\mathcal{S}$, the summation in~(\ref{bhdefu2}) vanishes.  We conclude that
\[
\langle (\dv \dv S)_{\rm dist} \df u, v \rangle_{V',V} = b_h(g; \df u, v) = -\langle \kappa u, dv \rangle_{W',W} = \langle \dv_{\rm dist} (\kappa u), v \rangle_{V',V}
\]
for all $v \in V$.
\end{proof}

\begin{remark} \label{remark:divdivSdef}
Proposition~\ref{lemma:dividistKu} implies in particular that for any smooth vector field $u$,
\begin{equation*} 
\dv \dv S \df u = \dv (\kappa u).
\end{equation*}
This can also be seen by considering the evolution of the curvature two-form $\kappa\, \omega$ under metric deformations induced by the flow $\varphi_t : \Omega \rightarrow \Omega$ of the vector field $u : \Omega \rightarrow \mathbb{R}^2$.  Indeed, consider the case where $u$ vanishes on $\partial\Omega$ for simplicity.  Let $g(t) = \varphi_t^* g(0)$ be a smooth family of Riemannian metrics on $\Omega$ obtained by pulling back $g(0)$ by $\varphi_t$.  Denote $\kappa(t)=\kappa(g(t))$ and $\omega(t)=\omega(g(t))$.  Using Proposition~\ref{prop:kappavoldot}, we see that
\begin{align*}
\dv (\kappa(0) u) \omega(0) = \mathcal{L}_u (\kappa(0)\omega(0)) &= \left.\frac{d}{dt}\right|_{t=0} \left( \kappa(t)\omega(t) \right) = \frac{1}{2} \left( \dv \dv S \dot{g}(0)\right) \omega(0) \\ &= \frac{1}{2} \left( \dv \dv S \mathcal{L}_u g(0) \right) \omega(0) 
= \left( \dv \dv S \df u \right) \omega(0),
\end{align*}
where the operators $\dv$, $\df$, and $S$ are taken with respect to $g(0)$.  See also~\cite[p. 13, Equation (1.28)]{chow2006hamilton}.
\end{remark}

\section*{Acknowledgments}
EG was supported by NSF grant DMS-2012427.

\appendix

\section{Appendix} \label{sec:appendix}

Below we verify that the sequence
\begin{center}
  \begin{tikzcd}
    0 \arrow{r}{} & V \arrow{r}{d} & W \arrow{r}{d} & X \arrow{r}{} & 0
  \end{tikzcd}
\end{center}
is exact on triangulations of contractible domains.  We do so by making three observations:
\begin{enumerate}
\item If $v \in V$ and $dv=0$, then clearly $v=0$ by the boundary conditions and the interelement continuity constraints imposed on functions in $V$.
\item The map $d : W \rightarrow X$ is surjective for the following reason.  On each $T \in \mathcal{T}_h$, the map
\[
\dv : H^1_0(T) \otimes \mathbb{R}^2 \rightarrow L^2_{\int=0}(T)
\]
is surjective~\cite[Lemma B.69, p. 492]{ern2004theory}, where $H^1_0(T) = \{f \in H^1(T) \mid f = 0 \text{ on } \partial T \}$ and $L^2_{\int=0}(T) = \{f \in L^2(T) \mid \int_T f \, \omega  = 0 \}$.  By rotating vectors $90^\circ$ and identifying them with one-forms, we see that
\[
d : H^1_0\Lambda^1(T) \rightarrow  L^2_{\int=0}\Lambda^2(T)
\]
is surjective, where $H^1_0\Lambda^1(T)$ denotes the space of one-forms on $T$ with coefficients in $H^1_0(T)$ and $L^2_{\int=0}\Lambda^2(T)$ denotes the space of square-integrable two-forms on $T$ with vanishing integral.  Now let $F \in X$ be arbitrary.  We can write $F = F_0 + F_1$, where $\int_T F_0$ vanishes on each $T \in \mathcal{T}_h$ and $F_1$ is piecewise constant.  The two-form $F_0$ is in the range of $d : W \rightarrow X$, since we can construct $\alpha_0 \in \prod_{T \in \mathcal{T}_h} H^1_0\Lambda^1(T) \subset W$ satisfying $d\alpha_0=F_0$ by above.   The two-form $F_1$ is also in the range of $d : W \rightarrow X$, since $d$ maps the Whitney one-forms {\color{blue}with vanishing trace} surjectively onto the piecewise constant two-forms {\color{blue}with vanishing mean}.  Thus $F$ is in the range of $d : W \rightarrow X$.
\item Now consider a one-form $\alpha \in W$ satisfying $d\alpha=0$.  We will show that there exists $v \in V$ such that $dv=\alpha$.  The canonical Whitney interpolant of $\alpha$, being closed, belongs to the range of $d : V \rightarrow W$; it is the image under $d$ of a continuous, piecewise linear function (a Whitney zero-form).  So it suffices to focus on the case where $\int_e \alpha = 0$ for every $e \in \mathcal{E}_h$.  On each triangle $T \in \mathcal{T}_h$, $\alpha\big|_T$ is a closed one-form belonging to $H^1\Lambda^1(T)$, so we can construct $v_T \in H^2(T)$ such that $dv_T = \alpha\big|_T$~\cite[Theorem 1.1]{costabel2010bogovskiui} and (by adding a suitable constant) $v_T$ vanishes at one of the vertices of $T$.  Since $\int_e dv_T = \int_e \alpha\big|_T = 0$ along each edge $e$ of $T$, $v_T$ in fact vanishes at every vertex of $T$.  On any edge $e$ shared by two triangles $T_1$ and $T_2$, the equality
\[
di_{T_1,e}^* v_{T_1} = i_{T_1,e}^* dv_{T_1} = i_{T_1,e}^* \alpha = i_{T_2,e}^* \alpha = i_{T_2,e}^* dv_{T_2} = di_{T_2,e}^* v_{T_2},
\]
together with the fact that $v_{T_1}$ and $v_{T_2}$ vanish at the endpoints of $e$, ensures that the trace of $v_{T_1}$ agrees with that of $v_{T_2}$ everywhere along $e$.  By similar reasoning, $v$ (the function whose restriction to $T$ is $v_T$ for each $T \in \mathcal{T}_h$) vanishes on edges $e \in \mathcal{E}_h \setminus \mathcal{E}_h^0$.  It follows that $v \in V$ and $\alpha=dv$.
\end{enumerate}

\printbibliography

\end{document}